\documentclass[9pt,reqno,twoside,final]{amsart}
\usepackage{amsmath,amstext,amsthm,amsfonts,amssymb,amscd}
\usepackage{tikz-cd}
\usepackage[mathscr]{eucal}
\usepackage{mathrsfs}
\numberwithin{equation}{section}
\makeatletter
\def\tagform@#1{\maketag@@@{$\langle$\ignorespaces#1\unskip\@@italiccorr$\rangle$}}

 \DeclareMathOperator{\End}{End}

\renewcommand{\phi}{\varphi}
\renewcommand{\epsilon}{\varepsilon}
\renewcommand{\kappa}{\varkappa}       
\renewcommand{\ge}{\geqslant}           

\newcommand{\Hom}{\mathop{\rm Hom}\nolimits}
\newcommand{\Ker}{\mathop{\rm Ker}\nolimits}

\DeclareMathOperator{\Aut}{Aut}
\DeclareMathOperator{\NAut}{NAut}
\DeclareMathOperator{\LKer}{LKer}

\DeclareMathOperator{\ad}{ad}
\DeclareMathOperator{\Deg}{deg}
\DeclareMathOperator{\Der}{Der}

\DeclareMathOperator{\Val}{Val}
\DeclareMathOperator{\Hal}{Hal}
\DeclareMathOperator{\Ind}{Ind}

\newtheorem{theorem}{Theorem}[section]

\newtheorem{rem}[theorem]{Remark}
\newtheorem*{rem*}{Remark}
\newtheorem*{acknow*}{Acknowledgements}
\newtheorem*{examples*}{Examples}

\theoremstyle{plain}

\newtheorem{prm}[theorem]{Problem}
\newtheorem{prop}[theorem]{Proposition}
\newtheorem*{theorem*}{Theorem}
\newenvironment{proof-sketch}{\noindent{\bf Sketch of Proof}\hspace*{1em}}{\qed\bigskip}
\newenvironment{proof-idea}{\noindent{\bf Proof Idea}\hspace*{1em}}{\qed\bigskip}
\newenvironment{proof-of-lemma}[1]{\noindent{\bf Proof of Lemma #1}\hspace*{1em}}{\qed\bigskip}
\newenvironment{proof-of-prop}[1]{\noindent{\bf Proof of Proposition #1}\hspace*{1em}}{\qed\bigskip}
\newenvironment{proof-of-thm}[1]{\noindent{\bf Proof of Theorem #1.}\hspace*{1em}}{\qed\bigskip}
\newenvironment{proof-attempt}{\noindent{\bf Proof Attempt}\hspace*{1em}}{\qed\bigskip}

\newtheorem{thm}[theorem]{Theorem}
\newtheorem{lem}[theorem]{Lemma}

\newtheorem{conj}[theorem]{Conjecture}
\newtheorem{Def}[theorem]{Definition}

\title[Some logical aspects]{Nonstandard analysis, deformation quantization and some logical aspects of (non)commutative algebraic geometry}

\author[Kanel-Belov]{Alexei Kanel-Belov}
\address[Alexei Kanel-Belov]{Bar-Ilan University, Ramat Gan, Israel}
\email{kanelster@gmail.com}
\author[Chilikov]{Alexei Chilikov}
\address[Alexei Chilikov]{Bauman Moscow State Technical University, Moscow, 105005 Russia;
Moscow Institute of Physics and Technology (National Research University), Dolgoprudnyi, Moscow Oblast, 141700 Russia}
\email{chilikov@passware.com}
\author[Ivanov-Pogodaev]{Ilya Ivanov-Pogodaev}
\address[Ilya Ivanov-Pogodaev]{Moscow Institute of Physics and Technology, Moscow, Russia}
\email{ivanov.pogodaev@gmail.com}
\author[Malev]{Sergey Malev}
\address[Sergey Malev]{Ariel University of Samaria, Ariel, Israel}
\email{sergeyma@ariel.ac.il}
\author[Plotkin]{Eugeny Plotkin}
\address[Eugeny Plotkin]{Bar-Ilan University, Ramat Gan, Israel}
\email{plotkin@math.biu.ac.il}
\author[Yu]{Jie-Tai Yu}
\address[Jie-Tai Yu]{College of Mathematics and Statistics, Shenzhen University, Shenzhen 518061, China}
\email{jietaiyu@163.com, jietaiyu@szu.edu.cn}
\author[Zhang]{Wenchao Zhang}
\address[Wenchao Zhang]{Bar-Ilan University, Ramat Gan, Israel}
\email{whzecomjm@gmail.com}
\thanks{The first three authors were supported by the Russian Science Foundation grant No. 17-11-01377,  fourth, fifth and seventh named authors were supported by the ISF (Israel Science Foundation) grant 1994/20.}
\thanks{We would like to thank the anonymous referee for suggestions regarding this paper.}

\begin{document}
\begin{abstract}
This paper surveys results related to  well-known works of B. Plotkin
 and V. Remeslennikov on the edge of algebra, logic and geometry.
We start from  a brief review of the paper and motivations. The first sections deal with model theory. In Section \ref{Pl}  we describe the geometric equivalence, the elementary equivalence, and the isotypicity of algebras. We look at these notions from the positions of universal algebraic geometry  and make emphasis  on the cases of the first order rigidity. In this setting Plotkin's problem on the structure of automorphisms of (auto)endomorphisms of free objects, and auto-equivalence of categories is pretty natural and important. Section \ref{plot-prm} is dedicated to particular cases of Plotkin's problem. Section \ref{indep} is devoted to Plotkin's problem for automorphisms of the group of polynomial symplectomorphisms. This setting has applications to mathematical physics through the use of model theory (non-standard analysis) in the studying of  homomorphisms between groups of symplectomorphisms and automorphisms of the Weyl algebra. The last two sections deal with algorithmic problems for noncommutative and commutative algebraic geometry. Section \ref{sec31} is devoted to the Gr\"obner basis in  non-commutative situation. Despite the existence of an algorithm for checking equalities, the zero divisors and nilpotency  problems  are algorithmically unsolvable. Section \ref{sec32} is connected with the problem of embedding of algebraic varieties; a sketch of the proof of its algorithmic undecidability over a field of characteristic zero is given.
\end{abstract}
\maketitle
\begin{flushright}
Dedicated to the 70-th anniversary of A.L. Semenov\\
and to the 95-th anniversary of B.I.~Plotkin.
\end{flushright}

\tableofcontents

\section{Introduction}\label{intro}
The connections between algebraic geometry and mathematical logic are extremely important.
First of all, notice a deep connection between algebra, category theory and model theory
 inspired by the results of Plotkin and his school
(see
\cite{Plotkin_UA-AL-Datab, Plotkin_Haz, Plotkin_7_lec, Plotkin_VarAlg-AlgVar, Pl-Sib}.
Note that this research is related to
 the one  of  most striking examples of interaction between model theory and geometry given by  solutions of the famous Tarskii's problem, see \cite{KM, Se}.
 Another outstanding  achievement is the theory of Zariski geometries developed by B.Zilber and E.Hrushovski \cite{HZ} and \cite{Z2,Z3}. In addition, the use of non-standard analysis has allowed progress in the theory of polynomial $64$ automorphisms. See the work of Belov and Kontsevich \cite{4,3}. For a detailed bibliography see \cite{b23}.


The foundations of algebraic geometry had an important part of the translation of topological and arithmetical properties into a purely algebraic language (\cite{Grothendieck}). 
Translation of the algebraic properties of a variety into the language of mathematical logic can be considered somehow in the spirit of this program.

 This short survey is related to ideas  contained in the 
 works of B. Plotkin and V. Remeslennikov and their followers. We assume that lots of questions  still require further illumination.

In particular, the following question is of interest: given two algebraic sets. Is there an algorithm for checking isomorphism? Similarly for birational equivalence. Is there a solution to the nesting problem of two varieties? For characteristics $0$, the answer is no \cite{Bel-Chil2}, but for a positive characteristic the answer is unknown.

On should mention a number of conjectures related to the theory of models and polynomial automorphisms expressed in the paper by Belov-Kontsevich \cite{4}. The investigations of the Plotkin school are far from completion, thus the relationship between the theory of medallions and the theory of categories is relevant.

 One of the goals of this paper is to narrow the gap and to draw attention to this topic. We deal with commutative and non-commutative algebraic geometry. The latter notion can be understood in several ways. There are many points of view on the subject. We touch universal algebraic geometry, some of its relations with reformational quantization and Gr\"obner basis in non-commutative situation.

The paper is organized as follows. {\bf Section \ref{mod-theor}} is devoted to {\bf various model-theoretical aspects and their applications}. More precisely, {\bf Section \ref{Pl}} deals with universal algebraic geometry and is focused around the interaction between algebra, logic, model theory and geometry. All these subjects are collected under the roof of the different kinds of logical rigidity of algebras. Under logical rigidity we mean some logical invariants of algebras whose coincidence gives rise to structural closeness of algebras in question. If such an invariant is strong enough then there is a solid ground to look for isomorphism of algebras whose logical invariants coincide.

We are comparing three types of logical description of algebras. Namely, we describe geometric equivalence of algebras, elementary equivalence of algebras and isotypicity of algebras. We look at these notions from the positions of universal algebraic geometry and logical geometry. This approach was developed by B. Plotkin and resulted in the consistent series of papers where algebraic logic, model theory, geometry and categories come together. In particular, an important role plays the study of automorphisms of categories of free algebras of the varieties. This question is highly related to description of such objects as $\Aut(\Aut)(A)$ and $\Aut(\End)(A)$, where $A$ is a free algebra in a variety.

We formulate the principal problems in this area and make a survey of the known results. Some of them are very recent while the others are quite classical. In any case we shall emphasize that we attract attention to the widely open important problem whether the finitely  generated isotypic groups are isomorphic.


The line started in  Section \ref{Pl} is continued in {\bf Section \ref{plot-prm}}.
 Problems related to universal algebraic geometry (i.e. algebraic geometry over algebraic systems) and
logical foundations of category theory gave rise to natural questions on automorphisms of categories and their auto-equivalences. The latter ones stimulate a new motivation to investigation of  semigroups of endomorphisms and groups of automorphisms of universal algebras (Plotkin's problem) (see \cite{Pl}).

Let $\Theta$ be a variety of linear algebras over a commutative-associative ring $K$ and $W=W(X)$ be a free algebra from $\Theta$  generated by a finite set $X$. Let
$H$ be an algebra from $\Theta$ and $AG_{\Theta}(H)$ be the
category of algebraic sets over $H$. Throughout the work, we refer to
\cite{b23, Pl-St} for definitions of the Universal Algebraic Geometry (UAG).

The category  $AG_{\Theta}(H)$ is considered as the logical invariant of an algebra $H$. By Definition \ref{def:gsim}, two algebras $H_{1}$ and $H_{2}$ are
geometrically similar if the categories $AG_{\Theta}(H_{1})$ and
$AG_{\Theta}(H_{2})$ are isomorphic.
It has been shown in \cite{Pl-St}, (cf., Proposition \ref{prop:eq}) that  geometrical similarity of algebras
is
determined by the structure of the group $\Aut(Theta^{0})$, where
$\Theta^{0}$ is the category of free finitely generated algebras
of $\Theta$.  The latter problem is treated by means of Reduction Theorem (see \cite{Pl-St}, \cite{KLP}, \cite{MPP2}, \cite{PZ1}). This theorem reduces investigation of automorphisms of the whole category $\Theta^0$ of free algebras in $\Theta$ to studying the  group $\Aut(\End(W(X)))$  associated with $W(X)$ in $\Theta^0$.

In Section \ref{plot-prm} we provide the reader with the results, describing $\Aut(\End(A))$, where $A$ is finitely generated free commutative or  associative algebra, over a field $K$. 


  We prove that the group $\Aut(\End(A))$ is generated by
semi-inner and mirror automorphisms of $\End(A)$ and the group
$\Aut(\mathcal{A}^{\circ})$ is generated by semi-inner and mirror
automorphisms of the category of free algebras $\mathcal{A}^{\circ}$.

 Earlier, the description of $\Aut(\mathcal{A}^{\circ})$  for the
 variety $\mathcal{A}$ of associative algebras over algebraically closed fields
 has been given in \cite{b19} and, over infinite fields, in \cite{b3}.
Also in the same works, the description of $\Aut(\End(W(x_{1},
x_{2})))$ has been obtained.

 Note that a description of the groups $\Aut(\End(W(X)))$ and $\Aut(\Theta^{\circ})$ for some other varieties
$\Theta$ has been given in \cite {b2, b3, b4, b8, b15, b16, b16a, b19,
b20, b21, MPP2, b27}.


A group of automorphisms of ind-schemes was computed in \cite{BEJ}. In investigating the Jacobian conjecture and automorphisms of the Weyl algebra, Plotkin's problem for symplectomorphisms is also extremely important. Such problems are associated with mathematical physics and the theory of $\mathcal{D}$-modules.


{\bf Section \ref{indep}} is devoted to {\bf  mathematical physics and model theory.}
This relation deals with nonstandard analysis.  We refer the reader to the review \cite{KanoveiLubec}.

The Belov--Kontsevich conjecture \cite{4}, sometimes Kanel-Belov--Kontsevich conjecture, dubbed $B{-}KKC_n$ for positive integer $n$, seeks to establish a canonical isomorphism between automorphism groups of algebras
\begin{equation*}
\Aut (A_{n,\mathbb{C}})\simeq \Aut(P_{n,\mathbb{C}}).
\end{equation*}
Here $A_{n,\mathbb{C}}$ is the $n$-th Weyl algebra over the complex field,
\begin{equation*}
A_{n,\mathbb{C}}=\mathbb{C}\langle x_1,\ldots,x_n,y_1,\ldots,y_n\rangle/(x_ix_j-x_jx_i,\; y_iy_j-y_jy_i,\; y_ix_j-x_jy_i-\delta_{ij}),
\end{equation*}
and $P_{n,\mathbb{C}}\simeq \mathbb{C}[z_1,\ldots,z_{2n}]$ is the commutative polynomial ring viewed as a $\mathbb{C}$-algebra and equipped with the standard Poisson bracket:
\begin{equation*}
\lbrace z_i,z_j\rbrace =\omega_{ij}\equiv \delta_{i,n+j}-\delta_{i+n,j}.
\end{equation*}
The automorphisms from $\Aut(P_{n,\mathbb{C}})$ preserve the Poisson bracket.

Let $\zeta_i,\;i=1,\ldots,2n$ denote the standard generators of the Weyl algebra (the images of $x_j,\;y_i$ under the canonical projection). The filtration by total degree on $A_{n,\mathbb{C}}$ induces a filtration on the automorphism group:
\begin{equation*}
\Aut^{\leq N}( A_{n,\mathbb{C}}):=\lbrace f\in\Aut(A_{n,\mathbb{C}})\;|\;\Deg f(\zeta_i),\;\Deg f^{-1}(\zeta_i)\leq N,\forall i=1,\ldots,2n\rbrace.
\end{equation*}
The obvious maps
\begin{equation*}
\Aut^{\leq N}( A_{n,\mathbb{C}})\rightarrow \Aut^{\leq N+1} (A_{n,\mathbb{C}})
\end{equation*}
are Zariski-closed embeddings, the entire group $\Aut (A_{n,\mathbb{C}})$ is a direct limit of the inductive system formed by $\Aut^{\leq N}$ together with these maps. The same can be said for the symplectomorphism group $\Aut(P_{n,\mathbb{C}})$.

The Belov--Kontsevich conjecture admits a stronger form, with $\mathbb{C}$ being replaced by the rational numbers. The latter conjecture will not be treated here in any way.

Since Makar-Limanov \cite{8}, \cite{9}, Jung \cite{6} and van der Kulk \cite{7}, the B-KK conjecture is known to be true for $n=1$. The proof is essentially a direct description of the automorphism groups. Such a direct approach however seems to be completely out of reach for all $n>1$. Nevertheless, at least one known candidate for isomorphism may be constructed in a rather straightforward fashion. The idea is to start with an arbitrary Weyl algebra automorphism, lift it after a shift by a certain automorphism of $\mathbb{C}$ to an automorphism of a larger algebra (of formal power series with powers taking values in the ring $^*\mathbb{Z}$ of hyperintegers) and then restrict to a subset of its center isomorphic to $\mathbb{C}[z_1,\ldots,z_{2n}]$.


This construction goes back to Tsuchimoto \cite{13}, who devised a morphism
 $\Aut (A_{n,\mathbb{C}})\rightarrow \Aut(P_{n,\mathbb{C}})$ in order to prove the stable equivalence between the Jacobian and the Dixmier conjectures. It was independently considered by Kontsevich and Kan\-el-Bel\-ov \cite{3}, who offered a shorter proof of the Poisson structure preservation which does not employ $p$-cur\-va\-tur\-es. It should be noted, however, that Tsuchimoto's thorough inquiry into $p$-cur\-va\-tur\-es has exposed a multitude of problems of independent interest, in which certain statements from the present paper might appear.


The construction we describe in detail in the following sections differs from that of Tsuchimoto in one aspect: an automorphism $f$ of the Weyl algebra may in effect undergo a shift by an automorphism of the base field $\gamma:\mathbb{C}\rightarrow\mathbb{C}$ prior to being lifted, and this extra procedure is homomorphic. Taking $\gamma$ to be the inverse nonstandard Frobenius automorphism (see below), we manage to get rid of the coefficients of the form $a^{[p]}$, with $[p]$ an infinite prime, in the resulting symplectomorphism. The key result here is that for a large subgroup of automorphisms, the so-called tame automorphisms, one can 
completely eliminate the dependence of the whole construction on the choice of the infinite prime $[p]$. Also, the resulting ind-group morphism $\phi_{[p]}$ is an isomorphism of the tame subgroups. In particular, for $n=1$ all automorphisms of $A_{1,\mathbb{C}}$ are tame (Makar-Limanov's theorem), and the map $\phi_{[p]}$ is the conjectured canonical isomorphism.


These observations motivate the question whether for any $n$ the group homomorphism $\phi_{[p]}$ is independent of infinite prime.


The next {\bf Section 3} makes emphasis on  {\bf algorithmic questions.}
First we dwell on 
 Non-Commutative Gr\"obner basis.
Questions of algorithmic decidability in algebraic structures have been studied since the 1940s.
In 1947 Markov \cite{Markov} and independently Post \cite{Post} proved that the word equality problem in finitely presented semigroups (and in algebras) cannot be algorithmically solved.
In 1952 Novikov constructed the first example of the group with unsolvable problem of word equality (see \cite{Novikov1} and \cite{Novikov2}).
In 1962 Shirshov proved solvability of the equality problem for Lie algebras with one relation
and raised a question about finitely defined Lie algebras \cite{Sh}.
In 1972 Bokut settled this problem.
In particular, he showed the existence of a finitely defined Lie algebra over an arbitrary field
with algorithmically unsolvable identity problem \cite{Bokut}.


Nevertheless, some problems become decidable if a finite Gr\"obner basis
defines a relations ideal. In this case it is easy to determine whether two elements of the algebra are equal or not (see \cite{Bergman2}).
In his work, D. Piontkovsky extended the concept of obstruction,
introduced by V. Latyshev (see \cite{Piont,Piont2,Piont3,Piont4}).
V.N. Latyshev raised the question concerning the existence
of an algorithm that can find out if a given element is either a zero divisor
 or a nilpotent element when the ideal of
relations in the algebra is defined by a finite Gr\"obner basis.

Similar questions for monomial automaton algebras can be solved.
In this case the existence of an algorithm for nilpotent element or a zero divisor was
proved by Kanel-Belov, Borisenko and Latyshev \cite{BBL}. Note that these algebras are not Noetherian and not weak Noetherian.
Iyudu showed that the element property of being one-sided zero divisor is recognizable
in the class of algebras with a one-sided limited processing (see \cite{Iudu}, \cite{IuduDisser}).
It also follows from a solvability of a linear recurrence relations system on a tree (see \cite{Belov}).


An example of an algebra with a finite Gr\"obner basis and algorithmically unsolvable problem
of zero divisor is constructed in \cite{IP}.

A notion of {\it Gr\"obner basis} (better to say {\it Gr\"obner-Shirshov basis})
first appeared in the context of noncommutative (and not Noetherian) algebra.
Note also that Poincar\'e-Birkhoff-Witt theorem can be canonically proved using Gr\"obner bases.
More detailed discussions of these questions see in \cite{Bokut}, \cite{Uf2}, \cite{BBL}.


To solve these two problems we simulate a universal Turing machine, each step of which corresponds to a multiplication from the left by a chosen letter.

The problem of the algorithmic decidability of the existence of an isomorphism between two algebraic varieties is extremely interesting and fundamental. A closely related problem is the embeddability problem. In the general form, it is formulated as follows.

{\bf Embeddability problem.}\ {\it  Let $\mathscr{A}$ and $\mathscr{B}$ be two algebraic varieties. Determine whether or not there exists an embedding of $\mathscr{A}$ in $\mathscr{B}$. Find an algorithm or prove its nonexistence.
}

In this paper, a negative solution to this problem is given even for affine varieties over an arbitrary field of characteristic zero whose coordinate rings are given by generators and defining relations.

\section{Model-theoretical aspects}\label{mod-theor}
Algebraic geometry over algebraic systems was investigated by B.I. Plotkin and his school. The section \ref{Pl} is devoted to this approach. In connection with this approach, Plotkin's problem about the automorphism of semigroups of endomorphisms of free algebra and categories (and also of groups of automorphisms) arose. The section \ref{plot-prm} is devoted to Plotkin's problem of endomorphisms and automorphisms. The problem of describing automorphisms for groups of polynomial symplectomorphisms and automorphisms of the Weyl algebra is extremely important, both from the point of view of mathematical physics and from the point of view of the Jacobian conjecture. Section \ref{indep} is dedicated to this problem.

\subsection{Algebraic geometry over algebraic systems}\label{Pl}

\subsubsection{\bf Three versions of logical rigidity}\label{Equiv}

Questions we are going to illuminate in this section are concentrated around the interaction between algebra, logic, model theory and geometry.

The main question behind further considerations is as follows. Suppose we have two algebras equipped with a sort of logical description.

\begin{prm} When the coincidence of logical descriptions  provides an isomorphism between algebras in question?
\end{prm}

With this aim we consider different kinds of logical equivalences between algebras. Some of the notions we are dealing with are not formally defined in the text.
For  precise definitions and references  use \cite{Halmos}, \cite{MR}, \cite{Plotkin_UA-AL-Datab},  \cite{Plotkin_Haz}, \cite{Pl-St}, \cite{PlAlPl}, \cite{PlPl}.

\medskip
\subsubsection{\bf Between syntax and semantics.}\label{synsem}
By syntax we will mean a language intended to describe a certain subject area.
In syntax we ask questions, express hypotheses and formulate the results. In syntax we
also build chains of formal consequences. For our goals we use first-order languages or their fragments. Each language is
based on some finite set of variables that serve as the alphabet, and a number of rules that
allow us to build words based on this alphabet. In general, its signature includes Boolean
operations, quantifiers, constants, and also functional symbols and predicate symbols. The
latter ones are included in atomic formulas and, in fact, determine the face of a particular
language. Atomic formulas will be called words. Words together with logical operations
between them will be called formulas.

By semantics we understand the world of models, or in other words, the subject
area of our knowledge. This world exists by itself, and develops according to its laws.

Fix a variety of algebras $\Theta$. Let $W(X)$, $X=\{x_1, \ldots, x_n\}$ denote the finitely generated  free algebra in $\Theta$. By equations in $\Theta$ we mean expressions of the form $w\equiv w'$, where $w$, $w'$ are words in $W(X)$ for some $X$. This is our first syntactic object.  Next, let $ \tilde \Phi=(\Phi(X), X\in \Gamma)$ be the multi-sorted Halmos algebra of first order logical formulas based on atoms $w\equiv w'$, $w$, $w'$ in $W(X)$, see \cite{Pl-St}, \cite{Plotkin_AGinFOL}, \cite{PlPl}. There is a special procedure to construct such an algebraic object which plays  the same role with respect to First Order Logic as Boolean algebras do with respect to Propositional calculus. One can view elements of  $\tilde\Phi=(\Phi(X), X\in \Gamma)$ just as first order formulas over $w\equiv w'$.


Let $X=\{x_1, \ldots , x_n \}$ and let $H$ be an algebra in the variety $\Theta$. We have an affine space $H^X$ of points  $\mu : X \to H$. For every $\mu$ we have also the $n$-tuple $(a_1, \ldots , a_n) = \bar a$ with $a_i = \mu(x_i)$. For the given $\Theta$ we have the homomorphism $$\mu : W(X) \to H$$ and, hence, the affine space is viewed as the set of homomorphisms
$$\Hom(W(X),H).$$
The classical kernel $\Ker(\mu)$ corresponds to each point $\mu : W(X) \to H$. This is exactly the set of equations for which the  point $\mu$ is a solution. Every point $\mu$ has also the logical kernel $\LKer(\mu)$, see \cite{PlAlPl}, \cite{Plotkin_7_lec},  \cite{Plotkin_AGinFOL}. Logical kernel $\LKer(\mu)$ consists of all formulas $u \in \Phi(X)$ valid on the point $\mu$. This is always an ultrafilter in $\Phi(X)$.


So we define syntactic and semantic areas where logic and geometry operate, respectively. Connect them by a sort of Galois correspondence.

 Let $T$ be a system of equations in $W(X)$. The set $A$ in the affine space $\Hom(W(X),H)$ consisting of all solutions of the system $T$ corresponds  to $T$. Sets of such kind are called {\it algebraic sets}. Vice versa, given a set $A$ of points in the affine space consider all equations $T$ having $A$ as the set of solutions. Sets $T$ of such kind are called {\it closed congruences} over $W$.

 We can do the same correspondence with respect to arbitrary sets of formulas. Given a set $T$ of formulas in algebra of formulas (set of elements) $\Phi(X)$, consider the set $A$ in the affine space, such that every point of $A$ satisfies every formula of $\Phi$. Sets of such
kind are called {\it definable sets}. Points of $A$ are called solutions of the set of formulas $T$.  Conversely,  given a set $A$ of points in the affine space consider all formulas (elements) $T$ having $A$ as the set of solutions. Sets $T$ of such kind are {\it closed filters} in  $\Phi(X)$.

Let us formalize the Galois correspondence described above.

\subsubsection{\bf Galois correspondence in the Logical Geometry}\label{Galois}


Let us start with a particular case when the set of formulas $T$ in $\Phi(X)$ is a set of equations of the form $w=w'$, $w, w' \in W(X)$, $X \in \Gamma$.


 We set
$$
A = T'_H = \{ \mu : W(X) \to H \ | \  T \subset Ker(\mu)\}.
$$
Here $A$ is an {\it algebraic set} in $\Hom(W(X),H)$, determined by the set $T$.

Let, further, $A$ be a subset in $\Hom(W(X),H)$. We set
$$
T = A'_H = \bigcap_{\mu \in A} Ker(\mu).
$$
Congruences $T$ of such kind are called  $H$-closed in $W(X)$. We have also Galois-closures $T''_H$ and $A''_H$.

Let us pass to the general case of logical geometry. Let now $T$ be a set of arbitrary formulas in $\Phi(X)$. We set
$$
A = T^L_H = \{ \mu : W(X) \to H \ | \  T \subset \LKer(\mu)\}.
$$
We have also
$$
A = \bigcap_{u \in T} \Val^X_H(u).
$$
Here $A$ is called a {\it definable  set} in $\Hom(W(X),H)$, determined by the set $T$. 
 We use the term "definable" for $A$ of such kind, meaning that $A$ is defined by some set of formulas $T$.

For the set of points $A$  in $\Hom(W(X),H)$ we set
$$
T = A^L_H = \bigcap_{\mu \in A} \LKer(\mu).
$$


We have also 
$$
T = A^L_H = \{ u\in \Phi(X) \ |\  A \subset \Val^X_H(u)\}.
$$

 Here $T$ is a Boolean filter in $\Phi(X)$ determined by the set of points $A$. Filters of such kind are Galois-closed and we can define the Galois-closures of arbitrary sets $T$ in $\Phi(X)$ and $A$ in $\Hom(W(X),H)$ as $T^{LL}$ and $A^{LL}$.


\begin{rem} The principal role in all considerations plays the value homomorphism $\Val: \tilde\Phi\to \Hal_\Theta$, where $\Hal_\Theta$ is a special Halmos algebra associated with the vector space $\Hom(W(X),H)$, see \cite{Plotkin_AGinFOL}, \cite{PlPl}. Its meaning is to make the procedure of verification whether a point satisfies the formula a homomorphism.
\end{rem}

\subsubsection{\bf Logical similarities of algebras}\label{lodsym}

Now we are in a position to introduce several logical equivalences between algebras. Since
the Galois correspondence  yields the duality between syntactic and semantic objects, every definition of equivalence between algebras formulated in terms of formulas, that is logically, has its semantical counterpart, that is a geometric formulation, and vice versa.

All algebraic sets constitute a category with special rational maps as morphisms \cite{PlPl}. The same is true with respect to definable sets \cite{PlPl}. So, we can formulate logical closeness of algebras geometrically.

\begin{Def}\label{def:gsim} We call algebras $H_1$ and $H_2$  {\it geometrically similar} if the categories of algebraic sets $AG_\Theta(H_1)$ and $AG_\Theta(H_2)$ are isomorphic.
\end{Def}

By Galois duality  between closed congruences and algebraic sets,   $H_1$ and $H_2$ are  geometrically similar if and only if the corresponding   categories $C_\Theta(H_1)$ and $C_\Theta(H_2)$ of closed congruences over $W(X)$  are isomorphic.

\begin{Def}\label{def:lsim} We call algebras $H_1$ and $H_2$  {\it logically similar}, if the categories of definable sets $LG_\Theta(H_1)$ and $LG_\Theta(H_2)$ are isomorphic.
\end{Def}

By Galois duality  between closed filters in $\Phi(X)$  and definable  sets,   $H_1$ and $H_2$ are  logically similar if and only if the corresponding categories   $F_\Theta(H_1)$ and $F_\Theta(H_2)$ of closed filters in $F(X)$  are isomorphic.

We will be looking for conditions $\mathcal A$ on algebras $H_1$ and $H_2$ that provide geometrical or logical similarity.

Let two algebras $H_1$ and $H_2$ subject to some condition $\mathcal A$ be given.  Here $\mathcal A$ is any condition of logical or,  dually, geometrical character, formulated in terms of closed sets of formulas or definable sets.

\begin{Def}\label{def:lsim2} We call the condition $\mathcal A$ rigid (or $\mathcal A$-rigid)  if two algebras $H_1$ and $H_2$  subject to  $\mathcal A$ are isomorphic.
\end{Def}

\subsubsection{\bf Geometric equivalence of algebras}\label{gme}

\begin{Def}\label{def:AG} {
Algebras $H_1$ and $H_2$ are called AG-equivalent, if for every $X $ and every system of equations $T$   holds       $T''_{H_1}=T''_{H_2}$.}
\end{Def}
 AG-equivalent algebras are called also {\it  geometrically equivalent} algebras, see \cite{PlPl}, \cite{PlAlPl}, \cite{Plotkin_7_lec}. The closure $T''_{H}$ is called, sometimes, a {\it radical} of $T$ with respect to $H$. This is a normal subgroup and an ideal in cases of groups and associative (Lie) algebras, respectively.

 The meaning of Definition \ref{def:AG} is as follows. Two algebras $H_1$ and $H_2$ are AG-equivalent if they have the same solution sets with respect to any system of equations $T$.  We have the following criterion, see \cite{PlPl}.

 \begin{prop}\label{:AG}
If algeras $H_1$ and $H_2$ are AG-equivalent, then they are AG-similar.
\end{prop}

So, geometric equivalence of algebras provides their geometrical similarity. The next statement describes geometrically equivalent algebras.  Assume, for simplicity, that our algebras are geometrically noetherian (see \cite{Plotkin_7_lec}, \cite{MR}),  which means that every system of equations $T$ is equivalent to a finite subsystem $T'$ of $T$.  Then, see \cite{Plotkin_Gagta},

\begin{prop}\label{quasi1}
Geometrically noetherian algebras $H_1$ and $H_2$ are AG-equivalent if and only if they generate the same quasi-variety.
\end{prop}

Hence, two algebras $H_1$ and $H_2$ are AG-equivalent if and only if they have the same quasi-identities. If we drop the condition of geometrical noetherianity, then algebras $H_1$ and $H_2$ are AG-equivalent if they have the same infinitary quasi-identities.

Let $\Theta$ be the variety of all groups. Now the question of $AG$-rigidity for groups reduces to the question when two groups generating one and the same  quasi-variety are isomorphic. Of course the condition on groups to have one and the same quasi-identities is  very weak   and the  rigidity of such kind  can happen if  both groups belong to a very narrow class of groups. In general, such a condition does not seem sensible.

Geometrical equivalence of algebras gives a sufficient condition for $AG$-similarity. It turns out that for some varieties $\Theta$ this condition is also sufficient.

\begin{thm}{}\label{th:gsim}
Let $Var(H_1)=Var(H_2)=\Theta.$ Let $\Theta$ be one of the following varieties


\begin{itemize}
\item
$\Theta=Grp$, the variety of groups,
 \item
 $\Theta=Jord$, the variety of Jordan algebras,
 \item
 $\Theta=Semi$, the variety of semigroups,
 \item  $\Theta=Inv$, the variety of inverse semigroups,
 \item  $\Theta=\mathfrak N_d$, the variety of nilpotent groups of class $d$.
 \end{itemize}
 Categories $AG_\Theta(H_1)$ and $AG_\Theta(H_2)$ are
isomorphic if and only if the algebras $H_1$ and $H_2$ are
geometrically equivalent (see \cite{Formanek}) \cite{MSZ}, \cite{Ts4},
\cite{Ts2}).

\end{thm}

Let $\Theta^0$ be the category of all free algebras of the variety $\Theta$.  The following proposition is the main tool in the proof of Theorem \ref{th:gsim}.

\begin{prop}[\cite{Pl-Sib}]\label{prop:eq} If for the variety $\Theta$ every automorphism of the category $\Theta^0$ is inner,
then  two algebras $H_1$ and $H_2$ are geometrically similar if and only if they are geometrically equivalent.
\end{prop}

So, studying automorphisms of $\Theta^0$ plays a crucial role.  The latter problem is treated by means of Reduction Theorem (see \cite{Pl-St}, \cite{KLP}, \cite{MPP2}, \cite{PZ1}). This theorem reduces investigation of automorphisms of the whole category $\Theta^0$ of free in $\Theta$ algebras  to studying the  group $\Aut(\End(W(X)))$  associated with a single object $W(X)$ in $\Theta^0$. Here, $W(X)$ is a finitely generated free in $\Theta$ hopfian algebra, which generates the whole variety $\Theta$. In fact, if all automorphisms of the endomorphism semigroup of a free algebra $W(X)$ are close to being inner, then all automorphisms of $\Theta^0$ possess the same property. More precisely, denote by $Inn(\End(W(X)))$ the group of inner automorphisms of $\Aut(\End(W(X)))$. Then the group of outer automorphisms $\Aut(\End(W(X)))/Inn(\End(W(X)))$ measures, in some sense, the difference between the notions of geometric similarity and geometric equivalence.

\medskip
\subsubsection{\bf Elementary equivalence of algebras}\label{ee}

As we saw in the previous section $AG$-equivalence of algebraic sets  reduces to coincidence of  quasi-identities of algebras.  This is a weak invariant, a small part of elementary theory, and, of course, coincidence of quasi-identities does not imply isomorphism of algebras.  Hence $AG$-equivalence does not make much sense from the point of view of rigidity. Now we recall a more powerful  logical invariant of  algebras.

Given algebra $H$, its {\it elementary theory} $Th(H)$ is the set of all sentences (closed formulas) valid on $H$. We modify a bit this definition and adjust it to the Galois correspondence. Fix $X=\{x_1, \ldots, x_n\}$. Define $X$-elementary theory $Th^X(H)$ to be
  the  set of all  formulas $u \in \Phi(X)$ valid in every point of the affine space $\Hom(W(X),H)$. In general we have a multi-sorted representation of the elementary theory $$Th(H)=(Th^X(H), X \in \Gamma),$$
where $\Gamma$ is a  certain  system of sets.

\begin{Def}\label{elem}
Two algebras $H_1$ and $H_2$ are said to be {\it elementarily equivalent} if their elementary theories coincide.
\end{Def}

\begin{rem} From the geometric point of view this definition does not make difference between different points of the affine space. Given algebras $H_1$ and $H_2$,  we collect all together formulas valid in every point of the affine spaces $\Hom(W(X),H_1)$ and  $\Hom(W(X),H_2)$, and  declare algebras  $H_1$ and $H_2$ elementarily equivalent if these sets coincide.
\end{rem}

  Importance of the elementary classification of algebraic structures goes back
to the famous  works of A.Tarski and A.Malcev. The main problem is to figure out {\it what are the algebras elementarily equivalent to a given one.} Very often we fix a class of algebras $\mathcal C$ and ask what are the algebras elementarily equivalent to a given algebra inside the class $\mathcal C$. So, the rigidity question with respect to elementary equivalence looks as follows.


 \begin{prm}
  Let a class of algebras $\mathcal C$ and an algebra $H\in \mathcal C$ be given. Suppose that the elementary theories of algebras $H$ and $A\in \mathcal C$ coincide. Are they  {\it elementarily rigid}, that is,  are  $H$ and $A$  isomorphic?
 \end{prm}

\begin{rem}
What we call elementary rigidity has different names. This notion appeared in the papers by A. Nies \cite{Nies}  under the name of quasi definability of groups. The corresponding name used in \cite{ALM} is first order rigidity. For some reasons which will be clear in the next section we use another term.
\end{rem}

In other words we ask for which algebras their logical characterization by means of the elementary theory is strong enough and define the algebra in the unique, up to an isomorphism, way?

 We restrict our attention to the case of groups, and, moreover, assume quite often that our groups are finitely generated.  Elementary rigidity of groups occurs not very often. Usually various extra conditions are needed. Here is the incomplete list of some known cases:

 \begin{thm}
 We will consider the following cases

 \begin{itemize}
 \item Finitely generated abelian groups are elementarily rigid in the class of such groups, see \cite{Szm}, \cite{EF}.
  \item Finitely generated torsion-free class 2 nilpotent groups are elementarily rigid in the class of finitely generated groups,  see \cite{Hirshon}, \cite{Oger} (this is wrong for such groups  of class 3 and for torsion groups of class 2, see \cite{Zilber1} ).
   \item If two finitely generated free nilpotent groups are elementarily equivalent, then they are isomorphic, that is a free finitely generated nilpotent group is elementarily rigid in the class of such groups,   see \cite{RSS}, \cite{Mal}.
   \item If two finitely generated free solvable groups are elementarily equivalent, then they are isomorphic, that is a free finitely generated solvable group is elementarily rigid in the class of such groups,   see \cite{RSS}, \cite{Mal}.
       \item Baumslag-Solitar group BS($1,n$) is elementarily rigid in the class of countable groups, see \cite{CKR}. General Baumslag-Solitar groups BS($m,n$) are elementarily rigid in the class of all Baumslag-Solitar groups, see  \cite{CKR}.
     \item Right-angled Coxeter group is elementarily rigid in the class of all right-angled Coxeter groups, see \cite{CK}.
      \item A good rigidity example is provided by profinite groups: if two finitely generated profinite groups are
elementarily equivalent (as abstract groups), then they are isomorphic \cite{JL}.
 \end{itemize}
  \end{thm}



Consider, separately, examples of elementary rigidity for linear groups. First of all, a group which is elementarily equivalent to a finitely generated linear group is a residually finite linear group \cite{Mal1}. The rigidity cases are collected in the following theorem.

 \begin{thm}
 We will consider the following cases.

 \begin{itemize}
 \item Historically, the first result was obtained by Malcev \cite{Mal2}. If two linear groups $GL_n(K)$ and  $GL_m(F)$, where $K$ and $F$ are    fields, are elementarily equivalent, then $n=m$ and the fields $K$ and $F$ are elementarily equivalent.
  \item This result was generalized to the wide class of Chevalley groups. Let $G_1=G_\pi(\Phi,R)$ and $G_2=G_\mu(\Psi,S)$ be two elementarily equivalent Chevalley groups. Here $\Phi$, $\Psi$ denote the root systems of rank $\geq 1$, $R$ and $S$ are local rings, and $\pi$, $\mu$ are weight lattices. Then root systems and weight lattices of $G_1$ and $G_2$ coincide, while the rings are elementarily equivalent. In other words Chevalley groups over local rings are elementarily rigid in the class of such groups modulo rigidity of the ground rings \cite{Bu}.
      \item Let $G_\pi(\Phi,K)$ be a simple Chevalley group over the algebraically closed field $K$. Then  $G_\pi(\Phi,K)$ is elementarily rigid in the class of all groups (cardinality is fixed). This result can be deduced from \cite{Zilber}. In fact, this is true for a much wider class of algebraic groups over algebraically closed fields and, modulo elementary equivalence of fields,  over arbitrary fields \cite{Zilber}.
   \item Any  irreducible non-uniform higher-rank characteristic zero
arithmetic lattice is elementarily rigid in the class of all groups, see \cite{ALM}. In particular, $SL_n(\mathbb Z)$, $n>2$ is elementarily rigid.
 \item Recently, the results of \cite{ALM} have been extended to a much more wide class of lattices, see \cite{AM}.
   \item Let $\mathcal O$ be the ring of integers of a number field, and let $n\geqslant 3$. Then every  group $G$ which is elementarily equivalent to $SL_n(\mathcal O)$ is isomorphic to $SL_n(\mathcal R)$, where the rings $\mathcal O$ and $\mathcal R$ are elementarily equivalent. In other words $SL_n(\mathcal O)$ is elementarily rigid in the class of all groups modulo elementary equivalence of rings. The similar results are valid with respect to $GL_n(\mathcal O)$ and  to the triangular group   $T_n(\mathcal O)$ \cite{SM}. These results intersect in part with the previous items, since  the ring $R=\mathbb Z$ is elementarily rigid in the class of all finitely generated rings \cite{Nies}, and thus  $SL_n(\mathbb Z)$ is elementarily rigid in the class of all finitely generated groups.
         \item  For the case of arbitrary Chevalley groups the results similar to  above cited are obtained in \cite{ST} by different machinery for a wide class of ground rings. Suppose  the Chevalley group $G=G(\Phi,R)$ of rank $\geqslant 2$ over the ring $R$ is given. Suppose that the ring $R$ is elementarily rigid in the class $\mathcal C$ of rings.  Then $G=G(\Phi,R)$ is elementarily rigid in the corresponding class $\mathcal C_1$ of groups  if $R$ is a field, $R$ is a local ring and $G$ is simply connected, $R$ is a Dedekind ring of  arithmetic type, that is the ring of $S$-integers of a number field, $R$ is Dedekind ring with at least 4 units and $G$ is adjoint. In particular, if  a ring of such kind is finitely generated then it gives rise to elementary rigidity of $G=G(\Phi,R)$ in the class of all finitely generated groups.
              If $R$ of such kind is not elementarily rigid then  $G=G(\Phi,R)$ is elementarily rigid in the class of all groups modulo elementary equivalence of rings.
 \end{itemize}
  \end{thm}

 Absolutely free groups lie on the other side of the scale of groups. It was Tarski who asked whether one can distinguish between finitely generated free groups by means of their elementary theories. This formidable  problem has been solved in affirmative, that is all free groups have one and the same elementary theory \cite{KM}, \cite{Se}. In fact, the variety of all groups is the only known variety of groups, such that a free in this variety finitely generated group is not rigid in the class of all such groups.

  \begin{prm} Construct a variety of groups different from the variety of all groups such that all free finitely generated groups in this variety have one and the same elementary theory.
  \end{prm}





\subsubsection{\bf Logical equivalence of algebras}\label{leq}

In this Section we introduce the notion of logical equivalence of algebras which can be viewed as first order equivalence. We proceed following exactly the same scheme which was applied in Section \ref{gme} with respect to the definition of geometric equivalence of algebras.

Let $H_1$ and $H_2$ be two algebras.  We will be looking for semantic logical invariant of these algebras, that is compare the definable sets over $H_1$ and $H_2$. Recall that according to Definition \ref{def:lsim} two algebras $H_1$ and $H_2$ are {\it logically similar}, if the categories of definable sets $LG_\Theta(H_1)$ and $LG_\Theta(H_2)$ are isomorphic.

Using the duality provided by Galois correspondence from Section \ref{Galois} we  will raise logical similarity to the  level of syntax.
 The principal Definition \ref{def:LG} is the first order counterpart of  Definition \ref{def:AG}.

\begin{Def}\label{def:LG}
Algebras $H_1$ and $H_2$ are called LG-equivalent (aka logically equivalent), if for every $X $ and every set of formulas $T$ in $\Phi(X)$   the  equality $T^{LL}_{H_1}=T^{LL}_{H_2}$ holds .
\end{Def}

It is easy to see that  
\begin{prop}
 If algebras $H_1$ and $H_2$ are
$LG$-equivalent then they are elementarily equivalent.
\end{prop}

Now we want to understand what is the meaning of logical equivalence.

\begin{Def}\label{def:iso}
 Two algebras $H_1$ and $H_2$ are called \textit{$LG$-isotypic} if for every point $\mu:W(X)\to H_1$ there exists a point $\nu:W(X)\to H_2$ such that $\LKer(\mu)=\LKer(\nu)$ and, conversely, for every point $\nu:W(X)\to H_2$ there exists a point $\mu:W(X)\to H_1$ such that $\LKer(\nu)=\LKer(\mu)$.
\end{Def}

The meaning of  Definition \ref{def:iso} is the following. Two algebras are isotypic if the sets of realizable types over $H_1$ and $H_2$ coincide. So, by some abuse of language these algebras have the same logic of types. Some references for the notion of isotypic algebras are contained in \cite{Plotkin_M}, \cite{Plotkin_Gagta}, \cite{PlAlPl}, \cite{PlPl}, \cite{PZ}, \cite{Zhitom_types}. Note that the notion was introduced in \cite{PZ}, \cite{Plotkin_M} while \cite{PlPl} gives the most updated survey.

The main theorem is as follows, see \cite{Zhitom_types}.
\begin{thm}\label{thm:lgiso} 
Algebras $H_1$ and $H_2$ are $LG$-equivalent if and only if they are LG-isotypic.
\end{thm}

Now we are in a position to study   rigidity of algebras with respect to isotypicity property. It is clear, that since isotypicity is stronger than elementary equivalence, this phenomenon can occur quite often. Let us state this problem explicitly.

 \begin{prm}
Let a class of algebras $\mathcal C$ and an algebra $H\in \mathcal C$ be given. Suppose that algebras $H\in \mathcal C$ and $A\in \mathcal C$ are isotypic. Are they  {\it isotypically rigid}, that is are $H$ and $A$  isomorphic?
 \end{prm}

\begin{rem}
In many papers from the list above isotypically rigid algebras are called {\it logically separable} \cite{PlPl}, \cite{Plotkin_Gagta}, or {\it type definable} \cite{MRoman}. 
\end{rem}

 \begin{thm}\label{isor}
 We will consider the following cases of rigidity:

 \begin{itemize}
 \item Finitely generated free abelian groups are isotypically rigid in the class of all groups, see \cite{Zhitom_types}.
 \item Finitely generated  free nilpotent groups of  class at most $n$  are isotypically rigid in the class of all groups \cite{Zhitom_types}.
    \item  Finitely generated   metabelian  groups   are isotypically rigid in the class of all groups \cite{MRoman}.
 \item Finitely generated  virtually polycyclic groups are isotypically rigid in the class of all groups \cite{MRoman}.
   \item  Finitely generated  free solvable   groups of derived length $d > 1$    are isotypically rigid in the class of all groups \cite{MRoman}.
       \item All surface groups, which are not  non-orientable surface
groups of genus 1,2 or 3   are isotypically rigid in the class of all groups \cite{MRoman}.
  \item Finitely generated absolutely free groups  are isotypically rigid in the class of all groups,  see \cite{Sklinos_1} based on \cite{Pillay} (also follows from \cite{PerinSklinos}, \cite{Houcine}, \cite{Zhitom_types_1}).
      \item Finitely generated free  semigroups are isotypically rigid in the class of semigroups, see \cite{Zhitom_types}.
       \item Finitely generated free  inverse semigroups are isotypically rigid in the class of inverse semigroups, see \cite{Zhitom_types}.
      \item{} Finitely generated free associative algebras are isotypically rigid in the class of such algebras \cite{Zhitom_types_1}.
 \end{itemize}
  \end{thm}

  The number of examples can be continued to co-Hopf groups, some Burnside groups, etc.

In fact, using either logical equivalence of algebras, or what is the same, the isotypicity of algebras, we compare the possibilities of individual points in the affine space  to define the sets of formulas (in fact ultrafilters in $\Phi(X)$) which are valid in these points. Given a point $\mu$ in the affine space,  the collection of formulas valid on the point $\mu$ is {\it a type} of $\mu$. If these individual types are,  roughly speaking, the same for both algebras,  then these algebras are declared isotypic. Thus, for isotypic algebras  we compare  types of formulas realizable on these  algebras. Of course, this is significantly stronger than elementary equivalence, where the individuality of points disappeared and we compare only formulas valid in all points of the affine space.

The following principal problem was stated in \cite{PlPl} and is widely open.

\begin{prm}[Rigidity problem]
 Is it true that every two isotypic finitely generated groups are
isomorphic?
\end{prm}

We will finish with the one more tempting problem of the same spirit.

\begin{prm}\label{fields}
What are the isotypicity classes of fields? When  two isotypic fields are isomorphic?
\end{prm}

The elementary equivalence of fields was one of motivating engines for Tarski to develop the whole model-theoretic staff related to elementary equivalence. Problem \ref{fields}, in a sense,  takes us back to the origins of the theory.

\subsection{Plotkin's problem: automorphisms of endomorphism semigroups and groups of polynomial automorphisms}
\label{plot-prm}

In the light of B.I. Plotkin's activity on  creation of algebraic geometry over algebraic systems, he drew a special attention to   studying the groups of their automorphisms, see \cite{Pl}. Later on he emphasized that automorphisms of categories of free algebras of the varieties play here a role of exceptional importance. This role was underlined in Proposition \ref{prop:eq} of Section \ref{Pl}.
The meaning  of Reduction Theorem (see \cite{Pl-St}, \cite{KLP}, \cite{MPP2}, \cite{PZ1}) was explained just after this proposition.  Reduction Theorem  reduces investigation of automorphisms of the whole category $\Theta^0$ of free in the variety $\Theta$ algebras  to studying the  group $\Aut(\End(W(X)))$  associated with a single object $W(X)$ in $\Theta^0$. Here, $W(X)$ is a finitely generated free in $\Theta$  algebra. In fact, if all automorphisms of the endomorphism semigroup of a free algebra $W(X)$ are close to being inner, then all automorphisms of $\Theta^0$ possess the same property.

  This philosophy forms a clear basis for  investigation of automorphisms of the semigroup of polynomial endomorphisms and the group of polynomial automorphisms.  The automorphisms of the endomorphism semigroup of a free associative algebra $A$ were given by Belov, Berzins and Lipyanski, (see  \cite{BDL} for details and definitions of semi-inner and mirror automorphisms):

\begin{thm}
The group ${\Aut}({\End}(A))$ is generated by semi-inner and mirror automorphisms of ${\End}(A)$.
Correspondingly, the group of automorphisms of the category of free associative algebras
is generated
by semi-inner and mirror automorphisms of this category.
\end{thm}

In the same spirit, the description of an endomorphism semigroup of the ring of commutative polynomials $A$ is given by Belov and Lipyanski in \cite{BL}:

 \begin{thm}
 Every automorphism of the group ${\Aut}({\End}(A))$ is
 semi-inner.
  \end{thm}

The automorphisms of the group of polynomial automorphisms on free associative algebras and commutative algebras at the level of {\it Ind}-schemes were obtained by Belov, Elishev and J.-T.Yu in  \cite{BEJ}. Let $K[x_1,\dots,x_n]$ and $K\langle x_1,\dots, x_n\rangle$ be  the free commutative polynomial algebra and the free associative algebra with $n$ generators, respectively. Denote by $\NAut$ the group of {\it nice} automorphisms, i.e., the group of  automorphisms which can
be approximated by tame ones. One can prove that  in characteristic zero case every
automorphism is nice.

 \begin{thm}
Any $\Ind$-scheme automorphism $\varphi$ of
$\NAut(K[x_1,\dots,x_n])$ for $n\ge 3$ is inner, i.e., it is a
conjugation via some automorphism of $K[x_1,\dots,x_n]$. Any $\Ind$-scheme automorphism $\varphi$ of
$\NAut(K\langle
x_1,\dots,x_n\rangle)$ for $n\ge 3$ is semi-inner (see \cite{BEJ} for the precise definition).
\end{thm}

\noindent
Here, the $\Ind$-scheme is defined as follows:

\begin{Def}
An {\it $\Ind$-variety} $M$ is the direct limit of algebraic varieties
$M=\varinjlim \lbrace M_1\subseteq M_2\cdots\rbrace.$
 An {\it $\Ind$-scheme} is an
$\Ind$-variety which is a group such that the group inversion is
a morphism $M_i\rightarrow M_{j(i)}$ of algebraic varieties, and the group multiplication induces
a morphism from $M_i\times M_j$ to $M_{k(i,j)}$. A map $\varphi$ is a {\it
morphism} of an $\Ind$-variety $M$ to an $\Ind$-variety $N$, if
$\varphi(M_i)\subseteq N_{j(i)}$ and the restriction $\varphi$ to
$M_i$ is a morphism for all $i$. Monomorphisms, epimorphisms and
isomorphisms are defined similarly in a natural way.
\end{Def}

\subsection{On the independence of the B-KK isomorphism of infinite prime and Plotkin conjecture for symplectomorphisms}
\label{indep}

\subsubsection{\bf Plotkin's problem for symplectomorphism and the Kontsevich conjecture}\label{plk}

Observe that the study of automorphisms of the group of polynomial symplectomorphisms, as well as automorphisms of the Weyl algebra (Plotkin's problem) is extremely important in course of the Kontsevich conjecture, as well as the Jacobian conjecture.

\subsubsection{\bf Ultrafilters and infinite primes}\label{uf}
Let $\mathcal{U}\subset 2^{\mathbb{N}}$ be an arbitrary non-principal ultrafilter on the set of all positive numbers (in this note $\mathbb{N}$ will almost always be regarded as the index set). Let $\mathbb{P}$ be the set of all prime numbers, and let $\mathbb{P}^{\mathbb{N}}$ denote the set of all sequences $p=(p_m)_{m\in\mathbb{N}}$ of prime numbers. We refer to a generic set $A\in \mathcal{U}$ as an index subset in situations involving the restriction $p_{|A}:A\rightarrow \mathbb{P}$. We will call a sequence $p$ of prime numbers $\mathcal{U}$-stationary if there is an index subset $A\in\mathcal{U}$ such that its image $p(A)$ consists of one point.

A sequence $p:\mathbb{N}\rightarrow\mathbb{P}$ is bounded if the image $p(\mathbb{N})$ is a finite set. Thanks to the ultrafilter finite intersection property, bounded sequences are necessarily $\mathcal{U}$-stationary.


Any non-principal ultrafilter $\mathcal{U}$ generates a congruence
\begin{equation*}
\sim_{\mathcal{U}}\subseteq \mathbb{P}^{\mathbb{N}}\times\mathbb{P}^{\mathbb{N}}
\end{equation*}
in the following way. Two sequences $p^1$ and $p^2$ are $\mathcal{U}$-congruent iff there is an index subset $A\in\mathcal{U}$ such that for all $m\in A$ the following equality holds:
\begin{equation*}
p^1_m=p^2_m.
\end{equation*}
The corresponding quotient
\begin{equation*}
{}^*\mathbb{P}\equiv \mathbb{P}^{\mathbb{N}}/\sim_{\mathcal{U}}
\end{equation*}
contains as a proper subset the set of all primes $\mathbb{P}$ (naturally identified with classes of $\mathcal{U}$-stationary sequences), as well as classes of unbounded sequences. The latter are referred to as nonstandard, or infinitely large, primes. We will use both names and normally denote such elements by $[p]$, mirroring the convention for equivalence classes. The terminology is justified, as the set of nonstandard primes is in one-to-one correspondence with the set of prime elements in the ring ${}^*\mathbb{Z}$ of nonstandard integers in the sense of Robinson \cite{R}.


Indeed, one may utilize the following construction, which was thoroughly studied\footnote{also cf. \cite{C}} in \cite{LLS}. Consider the ring $\mathbb{Z}^{\omega}=\prod_{m\in\mathbb{N}}\mathbb{Z}$ - the product of countably many copies of $\mathbb{Z}$ indexed by $\mathbb{N}$. The minimal prime ideals of $\mathbb{Z}^{\omega}$ are in bijection with the set of all ultrafilters on $\mathbb{N}$ (perhaps it is opportune to remind that the latter is precisely the Stone-Cech compactification $\beta\mathbb{N}$ of $\mathbb{N}$ as a discrete space). Explicitly, if for every $a=(a_m)\in\mathbb{Z}^{\omega}$ one defines the support complement as
\begin{equation*}
\theta(a)=\lbrace m\in\mathbb{N}\;|\;a_m=0\rbrace
\end{equation*}
and for an arbitrary ultrafilter $\mathcal{U}\in 2^{\mathbb{N}}$ sets
\begin{equation*}
(\mathcal{U})=\lbrace a\in\mathbb{Z}^{\omega}\;|\;\theta(a)\in\mathcal{U}\rbrace,
\end{equation*}
then one obtains a minimal prime ideal of $\mathbb{Z}^{\omega}$. It is easily shown that every minimal prime ideal is of such a form. Of course, the index set $\mathbb{N}$ may be replaced by any set $I$, after which one easily gets the description of minimal primes of $\mathbb{Z}^I$ (since those correspond to ultrafilters, there are exactly $2^{2^{|I|}}$ of them if $I$ is infinite and $|I|$ when $I$ is a finite set).
Note that in the case of finite index set all ultrafilters are principal, and the corresponding $(\mathcal{U})$ are of the form $\mathbb{Z}\times\cdots\times (0)\times\cdots\times\mathbb{Z}$ - a textbook example.

Similarly, one may replace each copy of $\mathbb{Z}$ by an arbitrary integral domain and repeat the construction above. If for instance all the rings in the product happen to be fields, then, since the product of any number of fields is von Neumann regular, the ideal $(\mathcal{U})$ will also be maximal.
\smallskip

The ring of nonstandard integers  may be viewed as a quotient (an ultrapower)
\begin{equation*}
\mathbb{Z}^{\omega}/(\mathcal{U})={}^*\mathbb{Z}.
\end{equation*}
The class of $\mathcal{U}$-congruent sequences $[p]$ corresponds to an element (also an equivalence class) in ${}^*\mathbb{Z}$, which may as well as $[p]$ be represented by a prime number sequence $p=(p_m)$, only in the latter case some but not too many of the primes $p_m$ may be replaced by arbitrary integers. For all intents and purposes, this difference is insignificant.

\smallskip
Also, observe that $[p]$ indeed generates a maximal prime ideal in ${}^*\mathbb{Z}$: if one for (any) $p\in [p]$ defines an ideal in $\mathbb{Z}^{\omega}$ as
\begin{equation*}
(p,\; \mathcal{U})=\lbrace a\in\mathbb{Z}^{\omega}\;|\;\lbrace m\;|\; a_m\in p_m\mathbb{Z}\rbrace\in\mathcal{U}\rbrace,
\end{equation*}
then, taking the quotient $\mathbb{Z}^{\omega}/(p,\; \mathcal{U})$ in two different ways, one arrives at an isomorphism
\begin{equation*}
{}^*\mathbb{Z}/([p])\simeq \left(\prod_{m}\mathbb{Z}_{p_m}\right)/(\mathcal{U}),
\end{equation*}
and the right-hand side is a field by the preceding remark.
\smallskip
For a fixed non-principal $\mathcal{U}$ and an infinite prime $[p]$, we will call the quotient
\begin{equation*}
\mathbb{Z}_{[p]}\equiv {}^*\mathbb{Z}/([p])
\end{equation*}
the nonstandard residue field of $[p]$. Under our assumptions this field has characteristic zero.

\medskip

\subsubsection{\bf Algebraic closure of nonstandard residue field}\label{ac}
We have seen that the objects $[p]$ - the infinite prime - behaves similarly to the usual prime number in the sense that a version of a residue field corresponding to this object may be constructed. Note that the standard residue fields are contained as a degenerate case in this construction, namely if we drop the condition of unboundedness and instead consider $\mathcal{U}$-stationary sequences, we will arrive at a residue field isomorphic to $\mathbb{Z}_{p}$, with $p$ being the image of the stationary sequence in the chosen class. The fields of the form $\mathbb{Z}_{[p]}$ are a realization of what is known as pseudofinite field, cf. \cite{BH}.

\smallskip

The nonstandard case is surely more interesting. While the algebraic closure of a standard residue field is countable, the nonstandard one itself has the cardinality of the continuum. Its algebraic closure is also of that cardinality and has characteristic zero, which implies that it is isomorphic to the field of complex numbers. We proceed by demonstrating these facts.
\smallskip

\begin{prop}
For any infinite prime $[p]$ the residue field $\mathbb{Z}_{[p]}$ has the cardinality of the continuum\footnote{There is a general statement on cardinality of ultraproduct due to Frayne, Morel, and Scott \cite{FMS}. We believe the proof of this particular instance may serve as a neat example of what we are dealing with in the present paper.}.
\end{prop}

\begin{proof}
It suffices to show there is a surjection

\begin{equation*}
h^*:\mathbb{Z}_{[p]}\rightarrow \mathfrak{P},
\end{equation*}

where $\mathfrak{P}=\lbrace 0,1\rbrace^{\omega}$ is the Cantor set given as the set of all countable strings of bits with the 2-adic metric

\begin{equation*}
d_2(x,y)=1/k,\;\;k=\min\lbrace m\;|\;x_m\neq y_m\rbrace.
\end{equation*}

The map $h^*$ is constructed as follows. If $\mathfrak{Z}\subset \mathfrak{P}$ is the subset of all strings with finite number of ones in them, and
\begin{equation*}
e:\mathbb{Z}_+\rightarrow\mathfrak{Z},\;\;e\left(\sum_{k<m}f_k2^k\right)=(f_1,\ldots,f_{m-1},0,\ldots)
\end{equation*}
is the bijection that sends a nonnegative integer to its binary decomposition, then for a class representative $a=(a_m)\in[a]\in\mathbb{Z}_{[p]}$ set $h^*(a)$ to be the (unique) ultralimit of the sequence of points $\lbrace x_m=e(a_m)\rbrace$. The correctness of this map rests on the property of the Cantor set being Hausdorff quasi-compact. Surjectivity is then established directly: consider an arbitrary $x\in\mathfrak{P}$. For each $m\in\mathbb{N}$ the set
\begin{equation*}
\mathfrak{P}_m=\lbrace e(0),e(1),\ldots,e(p_m-1)\rbrace
\end{equation*}
consists of $p_m$ distinct points. Let $x_m$ be the nearest to $x$ point from this set with respect to the 2-adic metric.
The sequence $(p_m)$ is unbounded, so that for every $m\in\mathbb{N}$ the index subset

\begin{equation*}
A_m=\lbrace k\in\mathbb{N}\;|\;p_k>2^m\rbrace
\end{equation*}
belongs to the ultrafilter $\mathcal{U}$. It is easily seen that for every $k\in A_m$ one has:

\begin{equation*}
d_2(x,x_k)<1/m
\end{equation*}
But that effectively means that the sequence $(x_m)$ has the ultralimit $x$, after which $a_m=e^{-1}(x_m)$ yields the desired preimage.
\end{proof}

As an immediate corollary of this proposition and the well-known Steinitz theorem, one has

\begin{thm}
The algebraic closure $\overline{\mathbb{Z}_{[p]}}$ of $\mathbb{Z}_{[p]}$ is isomorphic to the field of complex numbers.
\end{thm}
\smallskip

We now fix the notation for the aforementioned isomorphisms in order to employ it in the next section.

For any nonstandard prime $[p]\in {}^*\mathbb{P}$ fix an isomorphism $\alpha_{[p]}:\mathbb{C}\rightarrow \overline{\mathbb{Z}_{[p]}}$ coming from the preceding theorem. Denote by $\Theta_{[p]}:\overline{\mathbb{Z}_{[p]}}\rightarrow \overline{\mathbb{Z}_{[p]}}$ the nonstandard Frobenius automorphism - that is, a well-defined field automorphism that sends a sequence of elements to a sequence of their $p_m$-th powers:

\begin{equation*}
(x_m)\mapsto(x_m^{p_m}).
\end{equation*}

The automorphism $\Theta_{[p]}$ is identical on $\mathbb{Z}_{[p]}$; conjugated by $\alpha_{[p]}$, it yields a wild automorphism of complex numbers, as by assumption no finite power of it (as always, in the sense of index subsets $A\in\mathcal{U}$) is the identity homomorphism.

\bigskip

\subsubsection{\bf Extension of the Weyl algebra}\label{wa}
The $n$-th Weyl algebra $A_{n,\mathbb{C}}\simeq A_{n,\overline{\mathbb{Z}_{[p]}}}$ can be realized as a proper subalgebra of the following ultraproduct of algebras

\begin{equation*}
\mathcal{A}_n(\mathcal{U},[p])=\left(\prod_{m\in\mathbb{N}}A_{n,\mathbb{F}_{p_m}}\right)/\mathcal{U}.
\end{equation*}

Here for any $m$ the field $\mathbb{F}_{p_m}=\overline{\mathbb{Z}_{p_m}}$ is the algebraic closure of the residue field $\mathbb{Z}_{p_m}$. This larger algebra contains elements of the form $(\zeta^{I_m})_{m\in\mathbb{N}}$ with unbounded $|I_m|$ - something which is not present in $A_{n,\overline{\mathbb{Z}_{[p]}}}$, hence the proper embedding. Note that for the exact same reason (with degrees $|I_m|$ of differential operators having been replaced by degrees of minimal polynomials of algebraic elements) the inclusion

\begin{equation*}
\overline{\mathbb{Z}_{[p]}}\subseteq \left(\prod_{m\in\mathbb{N}}{F}_{p_m}\right)/\mathcal{U}
\end{equation*}
is also proper.
\smallskip

It turns out that, unlike its standard counterpart $A_{n,\mathbb{C}}$, the algebra $\mathcal{A}_n(\mathcal{U},[p])$ has a huge center described in this proposition:

\begin{prop}
The center of the ultraproduct of Weyl algebras over the sequence of algebraically closed fields $\lbrace \mathbb{F}_{p_m}\rbrace$ coincides with the ultraproduct of centers of $A_{n,\mathbb{F}_{p_m}}$:
\begin{equation*}
C(\mathcal{A}_n(\mathcal{U},[p]))=\left(\prod_{m}C(A_{n,\mathbb{F}_{p_m}})\right)/\mathcal{U}.
\end{equation*}
\end{prop}

The proof is elementary and is left to the reader. As in positive characteristic the center $C(A_{n,\mathbb{F}_{p}})$ is given by the polynomial algebra

\begin{equation*}
\mathbb{F}_{p}[x_1^{p},\ldots,x_n^{p},y_1^{p},\ldots,y_n^{p}]\simeq \mathbb{F}_{p}[\xi_1,\ldots,\xi_{2n}],
\end{equation*}

there is an injective $\mathbb{C}$-algebra homomorphism

\begin{equation*}
\mathbb{C}[\xi_1,\ldots\xi_{2n}]\rightarrow \left(\prod_{m}\mathbb{F}_{p_m}[\xi_1^{(m)},\ldots\xi_{2n}^{(m)}]\right)/\mathcal{U}
\end{equation*}

from the algebra of regular functions on $\mathbb{A}_{\mathbb{C}}^{2n}$ to the center of $\mathcal{A}_n(\mathcal{U},[p])$, evaluated on the generators in a straightforward way:

\begin{equation*}
\xi_i\mapsto[(\xi_i^{(m)})_{m\in\mathbb{N}}].
\end{equation*}

Just as before, this injection is proper.

Furthermore, the image of this monomorphism (the set which we will simply refer to as the polynomial algebra) may be endowed with the canonical Poisson bracket. Recall that in positive characteristic case for any $a,b\in \mathbb{Z}_p[\xi_1,\ldots,\xi_{2n}]$ one can define

\begin{equation*}
\lbrace a,b\rbrace=-\pi\left(\frac{[a_0,b_0]}{p}\right).
\end{equation*}

Here $\pi:A_{n,\mathbb{Z}}\rightarrow A_{n,\mathbb{Z}_p}$ is the modulo $p$ reduction of the Weyl algebra, and $a_0, b_0$ are arbitrary lifts of $a,b$ with respect to $\pi$. The operation is well defined, takes values in the center and satisfies the Leibnitz rule and the Jacobi identity. On the generators one has
\begin{equation*}
\lbrace \xi_i,\xi_j\rbrace=\omega_{ij}.
\end{equation*}
The Poisson bracket is trivially extended to the entire center $\mathbb{F}_{p}[\xi_1,\ldots,\xi_{2n}]$ and then to the ultraproduct of centers. Observe that the Poisson bracket of two elements of bounded degree is again of bounded degree, hence one has the bracket on the polynomial algebra.

\subsubsection{\bf Endomorphisms and symplectomorphisms}\label{es}

The point of this construction lies in the fact that thus defined Poisson structure on the (injective image of) polynomial algebra is preserved under all endomorphisms of $\mathcal{A}_n(\mathcal{U},[p])$ of bounded degree. Every endomorphism of the standard Weyl algebra is specified by an array of coefficients $(a_{i,I})$ (which form the images of the generators in the standard basis); these coefficients are algebraically dependent, but with only a finite number of bounded-order constraints. Hence the endomorphism of the standard Weyl algebra can be extended to the larger algebra $\mathcal{A}_n(\mathcal{U},[p])$. The restriction of any such obtained endomorphism on the polynomial algebra $\mathbb{C}[\xi_1,\ldots,\xi_{2n}]$ preserves the Poisson structure. In this setup the automorphisms of the Weyl algebra correspond to symplectomorphisms of $\mathbb{A}_{\mathbb{C}}^{2n}$.

\medskip

\textbf{Example}. If $x_i$ and $y_i$ are standard generators, then one may perform a linear symplectic change of variables:
\begin{eqnarray*}
f(x_i)=\sum_{j=1}^na_{ij}x_j+\sum_{j=1}^{n}a_{i,n+j}y_j,\;\;i=1,\ldots,n,\\
f(d_i)=\sum_{j=1}^na_{i+n,j}x_j+\sum_{j=1}^{n}a_{i+n,n+j}y_j,\;\; a_{ij}\in\mathbb{C}.
\end{eqnarray*}
In this case the corresponding polynomial automorphism $f^c$ of
\begin{equation*}
\mathbb{C}[\xi_1,\ldots,\xi_{2n}]\simeq\mathbb{C}[x_1^{[p]},\ldots,x_n^{[p]},y_1^{[p]},\ldots,y_n^{[p]}]
\end{equation*}
acts on the generators $\xi$ as
\begin{equation*}
f^c(\xi_i)=\sum_{j=1}^{2n}(a_{ij})^{[p]}\xi_j,
\end{equation*}
where the notation $(a_{ij})^{[p]}$ means taking the base field automorphism that is conjugate to the nonstandard Frobenius via the Steinitz isomorphism.

Let $\gamma:\mathbb{C}\rightarrow\mathbb{C}$ be an arbitrary automorphism of the field of complex numbers. Then, given an automorphism $f$ of the Weyl algebra $A_{n,\mathbb{C}}$ with coordinates $(a_{i,I})$, one can build another algebra automorphism using the map $\gamma$. Namely, the coefficients $\gamma(a_{i,I})$ define a new automorphism $\gamma_*(f)$ of the Weyl algebra, which is of the same degree as the original one. In other words, every automorphism of the base field induces a map $\gamma_*:A_{n,\mathbb{C}}\rightarrow A_{n,\mathbb{C}}$ which preserves the structure of the ind-object. It obviously is a group homomorphism.

\medskip

Now, if $P_{n,\mathbb{C}}$ denotes the commutative polynomial algebra with Poisson bracket, we may define an ind-group homomorphism $\phi:\Aut(A_{n,\mathbb{C}})\rightarrow \Aut(P_{n\mathbb{C}})$ as follows. Previously we had a morphism $f\mapsto f^c$, however as the example has shown it explicitly depends on the choice of the infinite prime $[p]$. We may eliminate this dependence by pushing the whole domain $\Aut(A_{n,\mathbb{C}})$ forward with a specific base field automorphism $\gamma$, namely $\gamma=\Theta_{[p]}^{-1}$ - the field automorphism which is Steinitz-conjugate with the inverse nonstandard Frobenius, and only then constructing the symplectomorphism $f^c_{\Theta}$ as the restriction to the (nonstandard) center. For the subgroup of tame automorphisms such as linear changes of variables this procedure has a simple meaning: just take the $[p]$-th root of all coefficients $(a_{i,I})$ first. We thus obtain a group homomorphism which preserves the filtration by degree and is in fact well-behaved with respect to the Zariski topology on $\Aut$ (indeed, the filtration $\Aut^{N}\subset \Aut^{N+1}$ is given by Zariski-closed embeddings). Formally, we have a proposition:

\begin{prop}
There is a system of morphisms
\begin{equation*}
\phi_{[p],N}:\Aut^{\leq N}(A_{n,\mathbb{C}})\rightarrow \Aut^{\leq N}(P_{n,\mathbb{C}}).
\end{equation*}
such that the following diagram commutes for all $N\leq N'$:
\begin{equation*}
\begin{tikzcd}
\Aut^{\leq N}(A_{n,\mathbb{C}}) \arrow{r}{\phi_{[p],N}} \arrow{d}{\mu_{NN'}} & \Aut^{\leq N}(P_{n,\mathbb{C}})\arrow{d}{\nu_{NN'}}\\
\Aut^{\leq N'}(A_{n,\mathbb{C}})\arrow{r}{\phi_{[p],N'}} & \Aut^{\leq N'}(P_{n,\mathbb{C}})
\end{tikzcd}
\end{equation*}


The corresponding direct limit of this system is given by $\phi_{[p]}$, which maps a Weyl algebra automorphism $f$ to a symplectomorphism $f^c_{\Theta}$.
\end{prop}
\medskip

The Belov -- Kontsevich conjecture then states:
\begin{conj}
$\phi_{[p]}$ is a group isomorphism.
\end{conj}
Injectivity may be established right away.
\begin{thm}
$\phi_{[p]}$ is an injective homomorphism.
\end{thm}
(See \cite{4} for the fairly elementary proof). 

\subsubsection{\bf On the loops related to infinite primes}\label{loops}
Let us at first assume that the Belov -- Kontsevich conjecture holds, with $\phi_{[p]}$ furnishing the isomorphism between the automorphism groups. This would be the case if all automorphisms in  $\Aut(A_{n,\mathbb{C}})$ were tame, which is unknown at the moment for $n>1$.


The main result of the paper is as follows:

\begin{thm}
If one assumes that $\phi_{[p],N}$ is surjective for any infinite prime $[p]$, then $\Phi_N$ is quasifinitedimensional and its eigenvalues are roots of unity.
\end{thm}
\smallskip

Let $[p]$ and $[p']$ be two distinct classes of $\mathcal{U}$-congruent prime number sequences - that is, two distinct infinite primes. We then have the following diagram:
\begin{equation*}
\begin{tikzcd}
\Aut(A_{n,\mathbb{C}}) \arrow{r}{\phi_{[p]}} \arrow{d}{isom} & \Aut(P_{n,\mathbb{C}})\arrow{d}{isom}\\
\Aut(A_{n,\mathbb{C}})\arrow{r}{\phi_{[p']}} & \Aut(P_{n,\mathbb{C}})
\end{tikzcd}
\end{equation*}
with all arrows being isomorphisms. Vertical isomorphisms answer to different presentations of $\mathbb{C}$ as $\overline{\mathbb{Z}_{[p]}}$ and $\overline{\mathbb{Z}_{[p']}}$. The corresponding automorphism $\mathbb{C}\rightarrow \overline{\mathbb{Z}_{[p]}}$ is denoted by $\alpha_{[p]}$ for any $[p]$.
\smallskip

The fact that all the arrows in the diagram are isomorphisms allows one instead to consider a loop of the form
\begin{equation*}
\Phi:\Aut(A_{n,\mathbb{C}})\rightarrow \Aut(A_{n,\mathbb{C}}).
\end{equation*}

Furthermore, as it was noted in the previous section, the morphism $\Phi$ belongs to \\
$\Aut(\Aut(A_{n,\mathbb{C}}))$.

We need to prove that $\Phi$ is a trivial automorphism. The first observation is as follows.

\begin{prop}
The map $\Phi$ is a morphism of algebraic varieties.
\end{prop}

\begin{proof}
Basically, this is a property of $\phi_{[p]}$ (or rather its unshifted version, $f_p\mapsto f^c_p$). More precisely, it suffices to show that, given an automorphism $f_p$ of the Weyl algebra in positive characteristic $p$ with coordinates $(a_{i,I})$, its restriction to the center (a symplectomorphism) $f^c_p$ has coordinates which are polynomials in $(a_{i,I}^p)$.

\smallskip

The switch to positive characteristic and back is performed for a fixed $f\in\Aut (A_{n,\mathbb{C}})$ on an index subset $A_f\in\mathcal{U}$.

\smallskip

Let $f$ be an automorphism of $A_{n,\mathbb{C}}$ and let $N=\Deg f$ be its degree. The automorphism $f$ is given by its coordinates $a_{i,I}\in\mathbb{C}$, $i=1,\ldots,2n$, $I=\lbrace i_1,\ldots,i_{2n}\rbrace$, obtained from the decomposition of algebra generators $\zeta_i$ in the standard basis of the free module:
\begin{equation*}
f(\zeta_i)=\sum_{i,I}a_{i,I}\zeta^I,\;\;\zeta^I=\zeta_1^{i_1}\cdots\zeta_{2n}^{i_{2n}}.
\end{equation*}

\medskip

Let $(a_{i,I,p})$ denote the class $\alpha_{[p]}(a_{i,I})$, $p=(p_m)$, and let $\lbrace R_k(a_{i,I}\;|\;i,I)=0\rbrace_{k=1,\ldots,M}$ be a finite set of algebraic constraints for coefficients $a_{i,I}$. Let us denote by $A_1,\ldots,A_M$ the index subsets from the ultrafilter $\mathcal{U}$, such that $A_k$ is precisely the subset, on whose indices the constraint $R_k$ is valid for $(a_{i,I,p})$. Take $A_f=A_1\cap\ldots\cap A_M\in\mathcal{U}$ and for $p_m$, $m\in A_f$, define an automorphism $f_{p_m}$ of the Weyl algebra in positive characteristic $A_{n,\mathbb{F}_{p_m}}$ by setting
\begin{equation*}
f_{p_m}(\zeta_i)=\sum_{i,I}a_{i,I,p_m}\zeta^I.
\end{equation*}
All of the constraints are valid on $A_f$, so that $f$ corresponds to a class $[f_p]$ modulo ultrafilter $\mathcal{U}$ of automorphisms in positive characteristic. The degree of every $f_{p_m}$ ($m\in A_f$) is obviously less than or equal to $N=\Deg f$.

\bigskip


Now consider $f\in\Aut^{\leq N}(A_{n,\mathbb{C}})$ with the index subset $A_f$ over which its defining constraints are valid. The automorphisms $f_{p_m}=f_p:A_{n,\mathbb{F}_p}\rightarrow A_{n,\mathbb{F}_p}$ defined for $m\in A_f\in\mathcal{U}$ provide arrays of coordinates $a_{i,I,p}$. Let us fix any valid $p_m=p$ denote by $\mathit{F}_{p^k}$ a finite subfield of $\mathbb{F}_p$ which contains the respective coordinates $a_{i,I,p}$ (one may take $k$ to be equal to the maximum degree of all minimal polynomials of elements $a_{i,I,p}$ which are algebraic over $\mathbb{Z}_p$).

\smallskip

Let $a_1,\ldots,a_s$ be the transcendence basis of the set of coordinates $a_{i,I,p}$ and let $t_1,\ldots,t_s$ denote $s$ independent (commuting) variables. Consider the field of rational functions:

\begin{equation*}
\mathit{F}_{p^k}(t_1,\ldots,t_s).
\end{equation*}
\medskip
The vector space
\begin{equation*}
\Der_{\mathbb{Z}_p}(\mathit{F}_{p^k}(t_1,\ldots,t_s),\mathit{F}_{p^k}(t_1,\ldots,t_s))
\end{equation*}
of all $\mathbb{Z}_p$-linear derivations of the field $\mathit{F}_{p^k}(t_1,\ldots,t_s)$ is finite-dimensional with \\$\mathbb{Z}_p$-dimension equal to $ks$; a basis of this vector space is given by elements $$\lbrace e_a\mathit{D}_{t_b}\;|\;a=1,\ldots, k,\;b=1,\ldots,s\rbrace$$ where $e_a$ are basis vectors of the $\mathbb{Z}_p$-vector space  $\mathit{F}_{p^k}$, and $\mathit{D}_{t_b}$ is the partial derivative with respect to the variable $t_b$.

\smallskip

Set $a_1,\ldots,a_s=t_1,\ldots,t_s$ (i.e. consider an $s$-parametric family of automorphisms), so that the rest of the coefficients $a_{i,I,p}$ are algebraic functions of $s$ variables $t_1,\ldots,t_s$. We need to show that the coordinates of the corresponding symplectomorphism $f_p^c$ are annihilated by all derivations $e_a\mathit{D}_{t_b}$.

Let $\delta$ denote a derivation of the Weyl algebra induced by an arbitrary basis derivation $e_a\mathit{D}_{t_b}$ of the field. For a given $i$, let us introduce the short-hand notation
\begin{equation*}
a=f_p(\zeta_i),\;\;b=\delta(a).
\end{equation*}
\medskip
We need to prove that
\begin{equation*}
\delta(f^c(\xi_i))=\delta(f_p(\zeta_i^p))=0.
\end{equation*}
In our notation $\delta(f_p(\zeta_i^p))=\delta(a^p)$, so by Leibnitz rule we have:
\begin{equation*}
\delta(f_p(\zeta_i^p))=ba^{p-1}+aba^{p-2}+\cdots+a^{p-1}b.
\end{equation*}
\smallskip

Let $\ad_x:A_{n,\mathbb{F}_p}\rightarrow A_{n,\mathbb{F}_p}$ denote a $\mathbb{Z}_p$-derivation of the Weyl algebra corresponding to the adjoint action (all Weyl algebra derivations are inner!):
\begin{equation*}
\ad_x(y)=[x,y].
\end{equation*}
We will call an element $x\in A_{n,\mathbb{F}_p}$ locally ad-nilpotent if for any $y\in A_{n,\mathbb{F}_p}$ there is an integer $D=D(y)$ such that
\begin{equation*}
\ad_x^D(y)=0.
\end{equation*}
All algebra generators $\zeta_i$ are locally ad-nilpotent. Indeed, one could take $D(y)=\Deg y + 1$ for every $\zeta_i$.

If $f$ is an automorphism of the Weyl algebra, then $f(\zeta_i)$ is also a locally ad-nilpotent element for all $i=1,\ldots,2n$. That means that for any $i=1,\ldots,2n$ there is an integer $D\geq N+1$ such that
\begin{equation*}
\ad_{f_p(\zeta_i)}^D(\delta(f_p(\zeta_i)))=\ad_{a}^D(b)=0.
\end{equation*}

\smallskip

Now, for $p\geq D+1$ the previous expression may be rewritten as
\begin{equation*}
0=\ad_{a}^{p-1}(b)=\sum_{l=0}^{p-1}(-1)^l\binom{p-1}{l}a^lba^{p-1-l}\equiv\sum_{l=0}^{p-1}a^lba^{p-1-l}\;\;(\text{mod\;} p),
\end{equation*}
and this is exactly what we wanted. \\
\smallskip

We have thus demonstrated that for an arbitrary automorphism $f_p$ of the Weyl algebra in characteristic $p$ the coordinates of the corresponding symplectomorphism $f_p^c$ are polynomial in $p$-th powers of the coordinates of $f_p$, provided that $p$ is greater than $\Deg f_p + 1$. As the sequence $(\Deg f_{p_m})$ is bounded from above by $N$ for all $m\in A_f$, we see that there is an index subset $A^*_f\in\mathcal{U}$ such that the coordinates of the symplectomorphism $f_{p_m}^c$ for $m\in A^*_f$ are polynomial in $p_m$-th powers of $a_{i,I,p_m}$. This implies that $f^c$ in characteristic zero is given by coefficients polynomial in $\alpha_{[p]}(a_{i,I})^{[p]}$ as desired.

\smallskip

It follows, after shifting by the inverse nonstandard Frobenius, that $\Phi$ is an endomorphism of the algebraic variety $\Aut (A_{n,\mathbb{C}})$.
\end{proof}

\medskip

The automorphism $\Phi$ acting on elements $f\in\Aut(A_{n,\mathbb{C}})$, takes the set of coordinates $(a_{i,I})$ and returns a set $(G_{i,I}(a_{k,K}))$ of the same size. All functions $G_{i,I}$ are algebraic by the above proposition. It is convenient to introduce a partial order on the set of coordinates. We say that $a_{i,I'}$ is higher than $a_{i,I}$ (for the same generator $i$) if $|I|<|I'|$ and we leave pairs with $i\neq j$ or with $|I|=|I'|$ unconnected. We define the dominant elements $a_{i,I}$ (or rather, dominant places $(i,I)$) to be the maximal elements with respect to this partial order, and subdominant elements to be the elements covered by maximal ones (in other words, for fixed $i$, subdominant places are the ones with $|I|=|I_{\max}|-1$).

\smallskip

The next observation follows from the fact that the morphisms in question are algebra automorphisms.

\begin{lem}
Functions $G_{i,I}$ corresponding to dominant places $(i,I)$ are identities:
\begin{equation*}
G_{i,I}(a_{k,K})=a_{i,I}.
\end{equation*}
\end{lem}
\begin{proof}
Indeed, it follows from the commutation relations that for any $i=1,\ldots,2n$ and $f_p$, $p=p_m,\; m\in A_f\in\mathcal{U}$, the highest-order term in $f_p^c(\xi_i)=f_p(\zeta_i^p)=f_p(\zeta_i)^p$ has the coefficient $a_{i,I,p}^p$. The shift by the inverse Frobenius then acts as the $p$-th root on the dominant place, so that we deduce that the latter is independent of the choice of $[p]$.
\end{proof}

\medskip

Let us now fix $N\geq 1$ and consider
\begin{equation*}
\Phi_N:\Aut^{\leq N} A_{n,\mathbb{C}}\rightarrow \Aut^{\leq N} A_{n,\mathbb{C}}
\end{equation*}
-- the restriction of $\Phi$ to the subvariety $\Aut^{\leq N} A_{n,\mathbb{C}}$, which is well defined by the above lemma. The morphism corresponds to an endomorphism of the ring of functions
\begin{equation*}
\Phi_N^*:\mathcal{O}(\Aut^{\leq N} A_{n,\mathbb{C}})\rightarrow \mathcal{O}(\Aut^{\leq N} A_{n,\mathbb{C}})
\end{equation*}

Let us take a closer look at the behavior of $\Phi_N$ (and of $\Phi^*_N$, which is essentially the same up to an inversion), specifically at how $\Phi_N$ affects one-dimensional subvarieties of automorphisms. Let $\mathcal{X}_N$ be the set of all algebraic curves of automorphisms in $\Aut^{\leq N} A_{n,\mathbb{C}}$; by virtue of Lemma 3.3 we may without loss of generality consider the subset of all curves with fixed dominant places -- we denote such a subset by $\mathcal{X}_N'$, and, for that same matter, the subsets $\mathcal{X}_N^{(k)}$ of curves with fixed places of the form $(i,I')$, which are away from a dominant place by a path of length at most $(k-1)$. In particular one has $\mathcal{X}_N'=\mathcal{X}_N^{(1)}$.

The morphism $\Phi_N$ yields a map
\begin{equation*}
\tilde{\Phi}_N:\mathcal{X}_N\rightarrow \mathcal{X}_N
\end{equation*}
and its restrictions
\begin{equation*}
\tilde{\Phi}_N^{(k)}:\mathcal{X}_N^{(k)}\rightarrow \mathcal{X}_N.
\end{equation*}
Our immediate goal is to prove that for all attainable $k$ we have
\begin{equation*}
\tilde{\Phi}_N^{(k)}:\mathcal{X}_N^{(k)}\rightarrow \mathcal{X}_N^{(k)},
\end{equation*}
i.e. the map $\Phi_N$ preserves the terms corresponding to non-trivial differential monomials.

In spite of minor abuse of language, we will call the highest non-constant terms of a curve in $\mathcal{X}_N^{(k)}$ dominant, although they cease to be so when that same curve is regarded as an element of $\mathcal{X}_N$.

Let $\mathcal{A}\in\mathcal{X}_N$ be an algebraic curve in general position. Coordinate-wise $\mathcal{A}$ answers to a set $(a_{i,I}(\tau))$ of coefficients parameterized by an indeterminate. By Lemma 3.3, $\Phi_N$ leaves the (coefficients corresponding to) dominant places of this curve unchanged, so we may well set $\mathcal{A}\in\mathcal{X}_N^{(1)}$. In fact, it is easily seen that the subdominant terms are not affected by $\Phi_N$ either, thanks to the commutation relations that define the Weyl algebra: for every $p$ participating in the ultraproduct decomposition, after one raises to the $p$-th power one should perform a reordering within the monomials -- a procedure which degrades the cardinality $|I|$ by an even number. Therefore, nothing contributes to the image of any subdominant term other than that subdominant term itself, which therefore is fixed under $\Phi_N$. We are then to consider the image

\begin{equation*}
\tilde{\Phi}_N^{(2)}(\mathcal{A})\in \mathcal{X}_N^{(2)}.
\end{equation*}

Again, given a positive characteristic $p$ within the ultraproduct decomposition, suppose the curve $\mathcal{A}$ (or rather its component answering to the chosen element $p$) has a number of poles attained on dominant\footnote{With respect to $\mathcal{X}_N^{(2)}$, i.e. the highest terms that actually change - see above where we specify this convention.} terms. Let us pick among these poles the one of the highest order $k$, and let $(i_0,I_0)$ be its place. By definition of an automorphism of Weyl algebra as a set of coefficients, the number $i_0$ does not actually carry any meaningful data, so that we are left with a pair $(k,|I_0|)$. As we can see, this pair is maximal from two different viewpoints; in fact, the pair represents a vertex of a Newton polygon taken over the appropriate field, with the discrete valuation given by $|I|$. The coordinate function $a_{i_0,I_0}$ corresponding to this pole admits a decomposition

\begin{equation*}
a_{i_0,I_0}=\frac{a_{-k}}{t^k}+\cdots,
\end{equation*}
with $t$ a local parameter. Acting upon this curve by the morphism $\Phi_N$ amounts to two steps: first, we raise everything to the $p$-th power and then assemble the components within the ultraproduct decomposition, then we take the preimage, which is essentially the same as taking the $p'$-root, with respect to a different ultraproduct decomposition. The order of the maximal pole is then multiplied by an integer during the first step and divided \emph{by the same integer} during the second one. By maximality, there are no other terms that might contribute to the resulting place in $\tilde{\Phi}_N^{(2)}(\mathcal{A})$. It therefore does not change under $\Phi_N$.

We may process the rest of the dominant (with respect to $\mathcal{X}_N^{(2)}$) terms similarly: indeed, it suffices to pick a different curve in general position. We then move down to $\mathcal{X}_N^{(k)}$ with higher $k$ and argue similarly.

\smallskip

After we have exhausted the possibilities with non-constant terms, we arrive at the conclusion that all that $\Phi_N$ does is permute the irreducible components of $\Aut^{\leq N} A_{n,\mathbb{C}}$. That in turn implies the existence of a positive integer $l$ such that

\begin{equation*}
\Phi_N^l=\text{Id}.
\end{equation*}

\medskip

In fact, the preceding argument gives us more than just the observation that $\Phi_N$ is unipotent.  Let $\Phi_{N,M}^*$ denote the linear map of finite-dimensional vector spaces obtained by restricting $\Phi_{N}^*$ to regular functions of total degree less than or equal to $M$. Then the following proposition holds.

\begin{prop}
If $\lambda$ is an eigenvalue of $\Phi_{N,M}^*$, then $\lambda^k=1$ for some integer $k$.
\end{prop}

\begin{proof}
Indeed, should there exist $\lambda_0\neq 1$, we may find an exceptional curve whose singularity changes under $\Phi_N$, note that coefficients are products of normalization coordinates.
\end{proof}




\subsubsection{\bf Discussion}\label{ds}
The investigation of decomposition of polynomial algebra-related objects into ultraproducts over the prime numbers $\mathbb{P}$ leads to a problem of independence of the choice of infinite prime. In the case of the Tsuchimoto -- Belov -- Kontsevich homomorphism the answer turns out to be affirmative, although there are other constructions, which are of algebraic or even polynomial nature but for which the independence fails. The reason for such arbitrary behavior has a lot to do with growth functions (in which case the situation is similar to the one described in\cite{17}, and in fact in\cite{5}, where one has a non-injective endomorphism $f_p:A_{n,\mathbb{F}_p}\rightarrow A_{n,\mathbb{F}_p}$, whose degree grows with $p$, which disallows for the construction of a naive counterexample to the Dixmier Conjecture in the ultralimit). It is, in our view, worthwhile to study such behavior in greater detail.

\section{Algorithmic aspects of algebraic geometry}\label{aloritm}
The section contains two subsections: the first one is devoted to noncommutative Finite Gr\"obner basis issues and the second one is devoted to algorithmic inclusion undecidability.
\subsection{Finite Gr\"obner basis algebras with unsolvable nilpotency problem and zero divisors problem}
\label{sec31}
\subsubsection{\bf The sketch of construction}\label{sk}

Let $A$ be an algebra over a field $K$.

The set of all words in the alphabet $\{a_1,\dots , a_N\}$ is a semigroup. The main idea of the construction is a
realization of a universal Turing machine in this semigroup. We use the universal Turing
machine constructed by Marvin Minsky in \cite{Minsky}. This machine has $7$ states and $4$-color
tape. The machine can be completely defined by $28$ instructions.

Note that $27$ of them have a form
$$(i,j) \rightarrow (L,q(i,j),p(i,j)) \text{ or } (i,j) \rightarrow (R,q(i,j),p(i,j)),$$
where $0\leq i \leq 6$ is the current machine state, $0\leq j\leq 3$ is the current cell color,
$L$ or $R$ (left or right) is the direction of a head moving after execution of the current
instruction, $q(i,j)$ is the state after current instruction, $p(i,j)$ is the new color of
the current cell.

Thus, the instruction $(2,3) \rightarrow (L,3,1)$ means the following: ``If the color of the
current cell is $3$ and the state is $2$, then the cell changes the color to $1$, the head moves
one cell to the left, the machine changes the state to $3$.

The last instruction is $(4,3) \rightarrow\ $STOP. Hence, if the machine is in state $4$ and
the current cell has color $3$, then the machine halts.

\medskip

\subsubsection*{\bf Letters.}\label{let}

By $Q_i$, $0\leq i\leq 6$ denote the current state of the machine. By $P_j$, $0\leq j\leq 3$
denote the color of the current cell.

The action of the machine depends on the current state $Q_i$ and current cell color $P_j$.
Thus every pair $Q_i$ and $P_j$ corresponds to one instruction of the machine.

The instructions moving the head to the left (right) are called {\it left} ({\it right}) ones.
Therefore there are {\it left pairs} $(i,j)$ for the left instructions, {\it right pairs} for
the right ones and instruction STOP for the pair $(4,3)$.


All cells with nonzero color are said to be {\it non-empty cells}.
We shall use letters $a_1$, $a_2$, $a_3$ for nonzero colors and letter $a_0$ for color zero.
Also, we use $R$ for edges of colored area.
Hence, the word $Ra_{u_1}a_{u_2}\dots a_{u_k}Q_iP_j a_{v_1}a_{v_2}\dots a_{v_l}R$ presents a full state of Turing machine.

We model head moving and cell painting using computations with powers of $a_i$ (cells) and $P_i$ and $Q_i$
(current cell and state of the machine's head).

\medskip


We use the universal Turing machine constructed by Minsky. This machine is defined by the following instructions:

\medskip

$(0,0)\rightarrow (L,4,1) \ (0,1)\rightarrow (L,1,3) \ (0,2)\rightarrow (R,0,0) \ (0,3)\rightarrow (R,0,1)$

$(1,0)\rightarrow (L,1,2) \ (1,1)\rightarrow (L,1,3) \ (1,2)\rightarrow (R,0,0) \ (1,3)\rightarrow (L,1,3)$

$(2,0)\rightarrow (R,2,2) \ (2,1)\rightarrow (R,2,1) \ (2,2)\rightarrow (R,2,0) \ (2,3)\rightarrow (L,4,1)$

$(3,0)\rightarrow (R,3,2) \ (3,1)\rightarrow (R,3,1) \ (3,2)\rightarrow (R,3,0) \ (3,3)\rightarrow (L,4,0)$

$(4,0)\rightarrow (L,5,2) \ (4,1)\rightarrow (L,4,1) \ (4,2)\rightarrow (L,4,0) \ (4,3)\rightarrow $ STOP

$(5,0)\rightarrow (L,5,2) \ (5,1)\rightarrow (L,5,1) \ (5,2)\rightarrow (L,6,2) \ (5,3)\rightarrow (R,2,1)$

$(6,0)\rightarrow (R,0,3) \ (6,1)\rightarrow (R,6,3) \ (6,2)\rightarrow (R,6,2) \ (6,3)\rightarrow (R,3,1)$

\medskip

We use the following alphabet:
$$\{t, \, a_0, \dots  a_3, \, Q_0, \dots Q_6, \, P_0 \dots P_3, \,  R\}$$

For every pair except $(4,3)$ the following functions are defined:
$q(i,j)$ is a new state, $p(i,j)$ is a new color of the current cell (the head leaves it).
\subsubsection{\bf Defining relations for the nilpotency question}\label{df1}

Consider the following defining relations:

\begin{eqnarray}
   &
tRa_l=Rta_l;  \text{\ \  $0\leq l\leq 3$}                           \label{tt1}  \\ &
ta_lR=a_lRt;  \text{\ \  $0\leq l\leq 3$}                           \label{tt1b}  \\ &
ta_ka_j=a_kta_j;                           \quad \text{ $0\leq k,j\leq 3$}    \label{tt2} \\ &
ta_kQ_iP_j=Q_{q(i,j)}P_kta_{p(i,j)}; \text{for left pairs $(i,j)$ and $0\leq k\leq 3$} \label{tt3} \\ &
tRQ_iP_j=RQ_{q(i,j)}P_0ta_{p(i,j)}; \text{for left pairs $(i,j)$ and $0\leq k\leq 3$} \label{tt5} \\ &
ta_lQ_iP_ja_ka_n=a_la_{p(i,j)}Q_{q(i,j)}P_kta_n;\text{for right pairs $(i,j)$ and $0\leq k\leq 3$} \label{tt4} \\ &
ta_lQ_iP_ja_kR=a_la_{p(i,j)}Q_{q(i,j)}P_kRt;\text{for right pairs $(i,j)$ and $0\leq k\leq 3$} \label{tt4r} \\ &
tRQ_iP_ja_ka_n=Ra_{p(i,j)}Q_{q(i,j)}P_kta_n;\text{for right pairs $(i,j)$ and $0\leq k\leq 3$} \label{tt4b} \\ &
tRQ_iP_ja_kR=Ra_{p(i,j)}Q_{q(i,j)}P_kRt;\text{for right pairs $(i,j)$ and $0\leq k\leq 3$} \label{tt4ar} \\ &
ta_lQ_iP_jR=a_la_{p(i,j)}Q_{q(i,j)}P_0Rt; \text{for right pairs $(i,j)$ and $0\leq l\leq 3$} \label{tt6}  \\ &
tRQ_iP_jR=Ra_{p(i,j)}Q_{q(i,j)}P_0Rt; \text{for right pairs $(i,j)$} \label{tt6b}  \\ &
Q_4P_3=0.                                  \label{tt7}
\end{eqnarray}
The relations \eqref{tt1} and \eqref{tt2} are used to move $t$ from the left edge to the last letter $a_l$ standing before
$Q_iP_j$ which represent the head of the machine.  The relations \eqref{tt3}--\eqref{tt6b}
represent the computation process. The relation \eqref{tt1b} is used to move $t$ through the finishing letter $R$.

Finally, the relation \eqref{tt7} halts the machine.

Let us call  $tRa_{u_1}a_{u_2}\dots a_{u_k}Q_iP_j a_{v_1}a_{v_2}\dots a_{v_l}R$ {\it the main word}.

\begin{thm} \label{th1}
Consider an algebra $A$ presented by the defining
relations \eqref{tt1}--\eqref{tt7}.
The word $tRUQ_iP_jVR$ is nilpotent in $A$ if and only if machine $M(i,j,U,V)$ halts.
\end{thm}

Actually we can prove that multiplication on the left by an element $t$ leads to the transition to the next state of the machine.

\subsubsection{\bf Defining relations for a zero divisors question}\label{df2}

We use the following alphabet:
$$\Psi=\{t, \, s, \, a_0, \dots  a_3, \, Q_0, \dots Q_6, \, P_0 \dots P_3, \,  L, \,  R \}.$$

For every pair except $(4,3)$ the following functions are defined:
$q(i,j)$ is a new state, $p(i,j)$ is a new color of the current cell (the head leaves it).


Consider the following defining relations:

\begin{eqnarray}
   &
tLa_k=Lta_k; \text{ $0\leq k\leq 3$}                             \label{td1}  \\ &
ta_ka_l=a_kta_l;                           \quad \text{ $0\leq k,l\leq 3$}    \label{td2} \\ &
sR=Rs;         \label{td9}  \\ &
sa_k=a_ks;    \text{ $0\leq k\leq 3$}  \label{td8} \\ &
ta_kQ_iP_j=Q_{q(i,j)}P_ka_{p(i,j)}s; \text{for left pairs $(i,j)$ and $0\leq k\leq 3$} \label{td3} \\ &
tLQ_iP_j=LQ_{q(i,j)}P_0a_{p(i,j)}s; \text{for left pairs $(i,j)$} \label{td5} \\ &
ta_lQ_iP_ja_k=a_la_{p(i,j)}Q_{q(i,j)}P_ks;\text{for right pairs $(i,j)$ and $0\leq k,l\leq 3$} \label{td4} \\ &
tLQ_iP_ja_k=La_{p(i,j)}Q_{q(i,j)}P_ks;\text{for right pairs $(i,j)$ and $0\leq k\leq 3$} \label{td4b} \\ &
ta_lQ_iP_jR=a_la_{p(i,j)}Q_{q(i,j)}P_0Rs; \text{for right pairs $(i,j)$ and  $0\leq l\leq 3$ } \label{td6}  \\ &
tLQ_iP_jR=La_{p(i,j)}Q_{q(i,j)}P_0Rs; \text{for right pairs $(i,j)$} \label{td6b}  \\ &
Q_4P_3=0;                                  \label{td7}
\end{eqnarray}
The relations \eqref{td1}--\eqref{td2} are used to move $t$ from the left edge to the letters
$Q_i$, $P_j$ which present the head of the machine.
The relations \eqref{td9}--\eqref{td8} are used to move $s$ from the letter $Q_i$, $P_j$ to the right edge.
The relations \eqref{td3}--\eqref{td6} represent the computation process. Here we use relations of the form $tU=Vs$.

Finally, the relation \eqref{td7} halts the machine.

\subsubsection{\bf Zero divisors and machine halt}\label{zd}

Let us call  $La_{u_1}a_{u_2}\dots a_{u_k}Q_iP_j a_{v_1}a_{v_2}\dots a_{v_l}R$ {\it the main word}.

\begin{thm} \label{th2}
The machine halts if and only if the main word is a zero divisor in
the algebra presented by the defining relations \eqref{td1}--\eqref{td7}.
\end{thm}

\medskip

\begin{rem*}
 We can consider two semigroups corresponding to our algebras:
 in both algebras each relation is written as an equality of two monomials.
 Therefore the same alphabets together with the same sets of relations define semigroups.
 In both semigroups the equality problem is algorithmically solvable, since it is solvable in algebras.
 However in the first semigroup a nilpotency problem is algorithmically unsolvable,
 and in the second semigroup a zero divisor problem is algorithmically unsolvable.
\end{rem*}

The entire proofs can be found at \cite{IPM}.

\subsection{On the algorithmic undecidability of the embeddability problem for algebraic varieties over a field of characteristic zero}
\label{sec32}

\subsubsection{\bf The Case of Real Numbers}\label{sec-real}
By a \textit{Matiyasevich family of polynomials} we mean a family of polynomials
$$Q(\sigma_1,\dots,\sigma_{\tau},x_1,\dots,x_s)$$
for which the existence of a solution for a given set of parameters of the polynomial is undecidable. As was established in \cite{Mat93}, such a polynomial exists.

Consider the affine space of dimension $5d+1$. We denote coordinates in this space by $X_i, Y_i, Z_i, U_i, W_i, 1\leq i\leq d$, and $T$. Consider the variety $B_{(d)}$ given by the following system of generators and relations:

\begin{equation}\label{sym-1}
\left\{\begin{array}{l}
X_{i}^{2}-\left(T^{2}-1\right) Y_{i}^{2}=1, \\
Y_{i}-(T-1) Z_{i}=V_{i}, \\
V_{i} U_{i}=1,
\end{array}\right.
\end{equation}

where $1\leq i\leq d$. For fixed $i$, the admissible values of the coordinates $X_i, Y_i, Z_i, U_i$, and $W_i$ are determined by the same value of $T$. Consider the ``short'' subsystem

\begin{equation}\label{sym-2}
\left\{\begin{array}{l}
X^{2}-\left(T^{2}-1\right) Y^{2}=1, \\
Y-(T-1) Z=V, \\
V U=1,
\end{array}\right.
\end{equation}

\begin{lem}
The following assertions hold for every solution of system (\ref{sym-2}):
\begin{enumerate}
    \item[1)] $U$ and $V$ are nonzero constants in $\mathbb{F}[t]$ $(\deg U=\deg V=0);$
    \item[2)] either $T=\pm 1$ and $X=\pm 1$ or
    $$Y=\sum_{k=0}^{[N / 2]}\left(\begin{array}{c}N \\2 k+1 \end{array}\right)\left(T^{2}-1\right)^{k} T^{N-1-2 k}$$
    for some integer $N$.
\end{enumerate}

\end{lem}

Let $R$ denote a root of the equation $R^2=T^2-1$ such that $R$ belongs to the algebraic extension $\mathbb{F}[t]$. Then the element $(T+R)^n$ can be uniquely represented in the form $X_n+R Y_n$, where $X_n$ and $Y_n$ are polynomials in $\mathbb{F}[t]$. All solutions of the equation

\begin{equation}
    X^2-(T^2-1)Y^2=1
\end{equation}
are of the form $X=\pm X_n$, $Y=\pm Y_n$ (see \cite{Den78}).

The structure of this set depends on $T$. In the case $T=\pm 1$, the first equation of the system imposes no conditions at all on $Y$. In turn, the other equation implies $Y=(T-1)Z+V$. For every choice of $V\in \mathbb{F}\setminus \{0\}$ and $Z\in \mathbb{F}[t]$, the corresponding solution exists and is unique.

\begin{lem}
If $\deg T>0$, then $V=Y \mathrm{mod} (T-1)=N$ for an integer $N$ and $Z=(Y-N)/(T-1).$ If $T=\text{const} \neq \pm 1$, then $Y$ and $Z$ are constants in $\mathbb{F}[t]$.
\end{lem}

Thus, the following three cases are possible:

\begin{enumerate}
    \item[(1)] for $\deg T>0$, to every set of integers $N_i$ there correspond polynomial solutions $Y_i$ and $X_i$ determined up to sign, as well as the constants $V_i=N_i$ and $U_i=1/V_i$, and $Z_i=(Y_i-V_i)/(T-1);$
    \item[(2)] for $\deg T=0$ and $T\neq \pm1$, there are constant solutions for $Y_i$ chosen from a given sequence; the values $X_i,Z_i,V_i$ and $U_i$ are also constants, and they are determined by the chosen values of $Y_i$;
    \item[(3)] for $T=\pm 1$, we obtain $X_i=\pm1$; for arbitrarily chosen constants $V_i$ and polynomials $Z_i$, we set $U_1=1/V_i$ and $Y_i=(T-1)Z_i+V_i.$
\end{enumerate}

So far, these considerations are valid for an arbitrary ground field $\mathbb{F}$ of characteristic zero. In the case $\mathbb{F}=\mathbb{R}$, we introduce a new coordinate $S$ by completing the main system of equations by the equation
\begin{equation}\label{sym-4}
    T=S^2+2,
\end{equation}
which ensures the impossibility of $T=\pm 1$. All common solutions of systems (\ref{sym-1}) and (\ref{sym-4}) either are constants (if $\deg T=0$, $T\neq\pm 1$) or correspond to some set of integer parameters $(N_1,\dots,N_d)$. We refer to solutions of the first kind as ``bad'' and to those of the second kind as ``good''.

Consider a Matiyasevich family of polynomials $Q(\sigma_1,\dots,\sigma_{\tau},x_1,\dots,x_s)$. Let $d\leq s$. Then, adding the new equation $Q(\sigma,V_1,\dots,V_s)=0$ to systems (\ref{sym-1}) and (\ref{sym-4}), we obtain a system defining a new variety. We denote this variety by $\mathscr{B}'_{(d),\sigma}$.

If $Q = 0$ has no integer solutions, then the original system has no good solutions. In this case, the variety $\mathscr{B}'_{(d),\sigma}$ is zero-dimensional, and there are no embeddings of $A$ in $\mathscr{B}'_{(d),\sigma}$.

Otherwise, for every solution $N_1,\dots,N_s$, we can explicitly construct functions $Y_i(S), X_i(S)$, and $Z_i(S)$ which are solutions. They define an embedding of the line in the variety $\mathscr{B}'_{(d),\sigma}$.

Since the existence of integer solutions for $Q$ is undecidable, it follows that so is the embeddability of $A$ in $\mathscr{B}'_{(d),\sigma}$ (in particular, in $\mathscr{B}'_{(s),\sigma}$). Here the input data is the equations defining $\mathscr{B}'_{(d),\sigma}$. We have proved the following theorem.

\begin{thm}
The problem of the embeddability of the affine line (and, therefore, the general embedding problem for an arbitrary variety) over $\mathbb{R}$ in an arbitrary algebraic variety $\mathscr{B}$ (defined by generators and relations) is undecidable.
\end{thm}

\subsubsection{\bf The Complex Case}\label{cc}
In this case, the situation is more complicated: it is hard to eliminate the case in which $T=\pm 1$ and $X_i=\pm 1$, since no constraints on $Y_i$ arise in this case. Therefore, we consider the problem of the embeddability of an affine space $A^m$ in a given variety $\mathscr{B}$ and construct a class of varieties such that it is impossible to decide whether a desired embedding exists from the defining relations for representatives of this class (for a certain suitable integer $m$). We define the coordinate ring of the variety $\mathscr{B}_{(d,e)}$ by the following system of generators and relations:

\begin{equation}\label{sym-5}
\left\{\begin{array}{l}
X_{i j}^{2}-\left(T_{j}^{2}-1\right) Y_{i j}^{2}=1, \\
Y_{i j}-\left(T_{j}-1\right) Z_{i j}=V_{i j}, \\
V_{i j} U_{i j}=1, \\
T_{j+1}=\prod_{k=1}^{j}\left(\left(T_{k}^{2}-1\right) W_{k}\right) W_{j+1},
\end{array}\right.
\end{equation}
where $1\leq i\leq d$ and $1\leq j\leq e$. In fact, we compose a system of many ``clones'' of the main system of the previous subsection and augment it by the ``linking'' relations between the parameters $T_j$. Let us study the solutions of the resulting system in $\mathbb{C}[t]$.

The relations for $X_{ij}, Y_{ij}, Z_{ij}, U_{ij}$, and $V_{ij}$ for each fixed $T_j$ are similar to those considered above. For a fixed set of $T_j$, the set of solutions is the direct sum of the sets $\mathscr{B}_{(d)}$, which have already been studied above.

As above, for each $j$, the following cases can occur: $T_j=\pm 1$ and $\deg T_j=0$; $T_j\neq\pm 1$, and $\deg T_j>0$.

The case most important from the point of view of ``elimination'' is the case where $T_{\hat{j}}=\pm 1$ for some $\hat{j}$. In this case, $T^2_{\hat{j}}-1=0$, and for all $j<\hat{j}$, we obtain
$$T_j=\prod_{k=1}^{j-1}((T_k^2-1)W_k)W_j.$$

\begin{lem}\label{lem-3}
If $T_N=C_N\neq 0$ for some $N$, then all $W_k$ with $k\leq N$ and all $T_k$ with $k\leq N-1$ are constants.
\end{lem}

By Lemma \ref{lem-3}, we have $T_j=C_j$ for $j<\hat{j}$. Here $C_j\neq\pm 1$ (otherwise $C_{j+1}=0$). Thus, if $T_{\hat{j}}=\pm 1$ for some $\hat{j}$, then the corresponding component has dimension $d$. However, in this case, all other components are zero-dimensional, and the total dimension of the variety does not exceed $d$.

In the second case, we have $T_{\hat{j}}=C_{\hat{j}}\neq \pm 1$ for some $\hat{j}$. The corresponding component of the variety has dimension 0. Moreover, Lemma \ref{lem-3} implies $T_j=C_j$ for $j<\hat{j}$. The corresponding $\hat{j}-1$ components of the variety are zero-dimensional as well.

The case $\deg T_j>0$ was considered in Sec. \ref{sec-real}. Each component of the variety is parametrized by a set of integers $N_{1j},\dots,N_{dj}$,for which the corresponding solutions for $X_{ij}, Y_{ij}, Z_{ij}, U_{ij}$, and $V_{ij}$ are constructed explicitly. The corresponding component has dimension 1.

Consider a Matiyasevich family of polynomials $Q(\sigma_1,\dots,\sigma_{\tau},x_1,\dots,x_s)$. The solvability problem of the Diophantine equation $Q(\sigma_1,\dots,\sigma_{\tau},V_{1j},\dots,V_{sj})=0$ is algorithmically undecidable. Let $d\leq s$. Adding the new equations $Q(\sigma,V_{i1},\dots,V_{is})=0$ to system \ref{sym-5}, we obtain a system defining a new variety. We denote it by $\mathscr{B}'_{(d,e),\sigma}$.

If $Q=0$ has no integer solutions, then the original system has no solutions for which $\deg T_0>0$. In this case, the possible solutions correspond either to the case where $T_j=\pm 1$ for some $j$ (and the set of solutions has dimension $d$) or to the case $T_j=C_j\neq\pm 1$. In the latter case, assuming that $j$ is the maximum index for which $T_j=C_j\neq \pm 1$, we see that all the succeeding $e-j$ components are one-dimensional and the total dimension of the set equals precisely $e-j\leq e-1$. Setting $e=s$ and $d=s-1$, we obtain

$$\dim \mathscr{B}'_{(d,e)}\leq \max (e-1,d)=s-1<s.$$

Obviously, in this case, for $m\geq s$, there is no embedding of $\mathscr{A}=A^m$ in $\mathscr{B}'_{(d,e),\sigma}=\mathscr{B}'_{(s-1,s),\sigma}.$ In particular, $A^s$ cannot be embedded in $\mathscr{B}'_{(s-1,s),\sigma}$.

If $Q$ has integer solutions, then, for every such solution $N_1,\dots,N_s$, we can explicitly construct functions $Y_{ij}(T), X_{ij}(T)$, and $Z_{ij}(T)$ which are solutions. These functions define an embedding of $A=A^s$ in the variety $\mathscr{B}'_{(d,e),\sigma}$.

Since the existence of integer solutions for $Q$ is undecidable, it follows that the embeddability of $A^s$ in $\mathscr{B}'_{(s-1,s),\sigma}$ is undecidable as well (the input data is the equations defining $\mathscr{B}'_{(s-1,s),\sigma}$). The proof is valid for any field $K$ of of characteristic zero. The following theorem holds.

\begin{thm}
There is a positive integer s for which the embeddability of an affine space $A^s$ over $K$ in an arbitrary algebraic variety $\mathscr{B}$ (defined by generators and relations) is undecidable. Thus, the general embeddability problem for an arbitrary algebraic variety is undecidable as well.
\end{thm}

\vspace{6pt}

\end{document}